\documentclass[english,plain]{article}
\pdfoutput=1
\RequirePackage{amssymb,amsmath,amsthm,bm,mathtools,amsfonts}
\RequirePackage{graphicx,xcolor,subfig}
\graphicspath{{./fig/}}

\usepackage{relsize}
\usepackage{fixmath} 
\usepackage{scalerel}
\newlength\bshft
\def\fakebold#1{\ThisStyle{\ooalign{$\SavedStyle#1$\cr%
\kern-\bshft$\SavedStyle#1$\cr%
\kern\bshft$\SavedStyle#1$}}}

\newcommand{\abs}[1]{\left\vert {#1} \right\vert}

\newcommand{\defeq}{\vcentcolon=}

\DeclareMathOperator{\dofop}{{\rm dof}}
\DeclareMathOperator{\diver}{{\rm div}}

\DeclareMathOperator{\rot}{{\rm rot}}
\DeclareMathOperator{\ROT}{{\mathbf{rot}}}
\newcommand{\gv}{\boldsymbol{g}}

\newcommand{\xx}{\boldsymbol{x}}
\newcommand{\xxp}{\boldsymbol{x}^{\perp}}
\newcommand{\mmp}{\boldsymbol{m}^{\perp}}
\newcommand{\uu}{\boldsymbol{u}}
\newcommand{\vv}{\boldsymbol{v}}

\newcommand{\ww}{\boldsymbol{w}}

\newcommand{\nn}{\boldsymbol{n}}
\newcommand{\VV}{\boldsymbol{V}}
\newcommand{\VVb}{{\protect\fakebold{{\mathbb{V}}}}}
\newcommand{\K}{\kappa}
\newcommand{\R}{\mathbb{R}}
\newcommand{\Pk}{\mathbb{P}}
\newcommand{\Pkt}{\widetilde{\mathbb{P}}}
\newcommand{\Mk}{\mathcal{M}}
\newcommand{\MMk}{\boldsymbol{\Mk}}
\newcommand{\Mkt}{\widetilde{\mathcal{M}}}
\newcommand{\Ie}{\mathfrak{e}}
\newcommand{\St}{\mathcal{S}^E}
\newcommand{\Ed}{\mathcal{E}_h}
\newcommand{\Edi}{\mathcal{E}_h^{\rm int}}
\newcommand{\Edb}{\mathcal{E}_h^{\rm ext}}
\newcommand{\EdE}{\mathcal{E}_h^{E}}
\newcommand{\Ede}{\mathcal{E}_h^{\partial_e \Omega}}
\newcommand{\Edn}{\mathcal{E}_h^{\partial_n \Omega}}

\newcommand\blfootnote[1]{%
  \begingroup
  \renewcommand\thefootnote{}\footnote{#1}%
  \addtocounter{footnote}{-1}%
  \endgroup
}

\newtheorem{lemma}{Lemma}[section]
\newtheorem{prop}{Proposition}[section]
\newtheorem{remark}{Remark}[section]
\newtheorem{problem}{Problem}
\newtheorem{dof}{Degrees of freedom}

\newtheorem{property}{Property}
\newtheorem{assumption}{Assumption}
\usepackage[all,cmtip]{xy}

\title{The Mixed Virtual Element Method\\ on curved edges in two dimensions}

\author{Franco Dassi$^*$  \and Alessio Fumagalli$^\dagger$ \and Davide Losapio$^\dagger$ \and Stefano Scial\`o$^\diamondsuit$ \and Anna Scotti $^\dagger$ \and  Giuseppe Vacca$^*$}

\begin{document}

\maketitle    

\begin{abstract}
    In this work, we propose an extension of the mixed Virtual Element Method
    (VEM) for bi-dimensional computational grids with curvilinear edge elements.
    The approximation by means of rectilinear edges of a domain with curvilinear
    geometrical feature, such as a portion of domain boundary or an internal
    interface, may introduce a geometrical error that degrades the expected
    order of convergence of the scheme. In the present work a suitable VEM
    approximation space is proposed to consistently handle curvilinear
    geometrical objects, thus recovering optimal convergence rates.  The
    resulting numerical scheme is presented along with its theoretical analysis
    and several numerical test cases to validate the proposed approach.
\end{abstract}

\blfootnote{$^*$Dipartimento di Matematica e Applicazioni, Universit\`a degli Studi di Milano Bicocca, Via Roberto Cozzi 55 - 20125 Milano, Italy. \texttt{franco.dassi@unimib.it  } \texttt{giuseppe.vacca@unimib.it} \\
$^\dagger$ MOX - Dipartimento di Matematica, Politecnico di Milano, via Bonardi 9, 20133 Milano, Italy \texttt{davide.losapio@polimi.it   }\texttt{alessio.fumagalli@polimi.it   }\texttt{anna.scotti@polimi.it}\\
$^\diamondsuit$ Dipartimento di Scienze Matematiche, Politecnico di Torino, Corso Duca degli Abruzzi 24, 10129 Torino, Italy \texttt{stefano.scialo@polito.it}}

\section{Introduction}

The present work proposes an extension of the Mixed Virtual Element meth\-od for
meshes with elements having curved edges, for bi-dimensional elliptic
problems in mixed form. The method allows to handle domains with curved
boundaries, or domains with embedded curved interfaces, or even  with mesh elements having all curved edges.

Mixed methods are well suited for the discretization of  vector field in
$H(\text{div})$. Classes of mixed methods, in addition to the here considered
Mixed Virtual Element Methods (MVEM)~\cite{BeiraodaVeiga2014b} are the well
known Raviart-Thomas (RT) ~\cite{Raviart1977,Roberts1991,Arnold2005,Boffi2013}
and Brezzi-Douglas-Marini
(BDM)~\cite{Brezzi1985,Nedelec1986,Brezzi1987,Boffi2013} finite element schemes.

Dealing with curved boundaries/interfaces for approximation degrees greater than
one has some pitfalls. Indeed, as the polynomial accuracy increases, the
geometrical error due to the approximation of curved boundaries or interfaces
via piece-wise linear edges dominates the numerical error of the scheme, thus
bounding the convergence rate.

Curved edge discretizations have been investigated in the virtual element
framework for the first time in~\cite{BeiraodaVeiga2019} where an elliptic
bi-dimensional problem in primal formulation is considered. The proposed
approach is based on standard virtual elements (VEM) and it is well suited for
problems where the computational domain is characterized by fixed curved
boundaries or interfaces. After this pioneering work, other strategies have been
proposed to extended the VEM to curved edge elements.  In~\cite{Bertoluzza2019},
for example, the authors keep the standard definition of VEM spaces and suggest
properly modified bi-linear forms to take into account elements with curved
boundaries. In~\cite{brezziCurvo} the virtual element space proposed
by~\cite{BeiraodaVeiga2019} is modified to contain polynomials. Such extension
is crucial to preserve convergence rates when the mesh element diameter
decreases while boundary curvature remains fixed.  

Other important classes of
methods have been developed to handle curved edges: isogeometric
analysis~\cite{Hughes2005,Bazilevs2006,Montardini2017}, non-affine isoparametric
elements~\cite{Ciarlet1972,Zlamal1973,Lenoir1986} and also Mimetic Finite
Differences~\cite{Brezzi:2006} or Hybrid High Order schemes~\cite{dipietro}.

The main advantage of VEM based approaches for curved edge elements lies in the
possibility of exactly reproducing the curved interface or domain boundary
without introducing any geometrical approximation, provided a suitable
parametric description of the curve is available. The use of mixed
discretizations, further, is particularly well suited for problems where local
mass conservation is of paramount importance. These features, combined with the
great flexibility of VEM make the proposed approach
particularly well suited for single or multi-phase flow problems in
heterogeneous porous media, with or without the presence of fractures, for the
analysis of absorbing materials, or composite materials, or materials with
inclusions of arbitrary shapes. Indeed,  these applications are characterized by
complex domains with arbitrary shape interfaces, or multiple intersecting
interfaces, and coefficients with strong variations. Examples of applications of
the MVEM with rectilinear edge meshes can be found, e.g.
in~\cite{BeiraoVeiga2016,Benedetto2016,Fumagalli2016a,Fumagalli2017,Benedetto2017,Fumagalli2017a,Dassi2019,hyperVEM,pichlerVEM,Fumagalli2020b}.

Here the MVEM is extended to curved edge elements following the approach proposed
in~\cite{BeiraodaVeiga2019}.  Concerning the choice of the degrees of
freedom, the proposed scheme can be seen as a generalization of RT elements to
order $k\geq 0$ to curved edges. Moreover, the new scheme is an extension to
the curved case of the classical virtual mixed spaces.  Indeed, when the domain has
no curved boundaries or interfaces, the proposed virtual spaces boil down to the
spaces defined in~\cite{BeiraodaVeiga2014b,BeiraoVeiga2016}, with a slightly
different choice of the degrees of freedom, that is particularly suited for curved elements.

{The paper is organized as follows.
In Section~\ref{sec:mathematical_model} we discuss some technical details as
notations, mathematical model and hypothesis on curved edges.
In Section~\ref{sec:discrete_spaces} we present the mesh assumptions, introduce the
discrete spaces, with the associated set of degrees of freedom and define the
discrete bilinear forms, then we present the discrete problem.  In
Section~\ref{sec:theory} we analyse the theoretical properties of the proposed
method: we  introduce the Fortin operator, we establish the discrete inf-sup
condition and provide the interpolation estimate for the curved MVEM. Then we
prove the stability bounds for the associated discrete bilinear form. 
At the end of this section we recover the optimal order of convergence of the present method. 
In Section~\ref{sec:numExe} we provide some experiments to give numerical evidence
of the behaviour
of the proposed scheme. Finally, Section \ref{sec:conclusion} is devoted to
conclusion.}

\section{Notations and Preliminaries}\label{sec:mathematical_model}

{Throughout the paper, we will follow the usual notation for Sobolev spaces
and norms as in \cite{Adams:1975}.
Hence, for a bounded domain $\omega$,
the norms in the spaces $W^s_p(\omega)$ and $L^p(\omega)$ are denoted by
$\|{\cdot}\|_{W^s_p(\omega)}$ and $\|{\cdot}\|_{L^p(\omega)}$ respectively.
Norm and seminorm in $H^{s}(\omega)$ are denoted respectively by
$\|{\cdot}\|_{s,\omega}$ and $|{\cdot}|_{s,\omega}$,
while $(\cdot,\cdot)_{\omega}$ and $\|\cdot\|_{\omega}$ denote the $L^2$-inner product and the $L^2$-norm (the subscript $\omega$ may be omitted when $\omega$ is the whole computational
domain $\Omega$).
Moreover with a usual notation, the symbols $\nabla$, $\Delta$ denote the gradient and Laplacian for scalar functions, while $\diver$ denotes the divergence for vector fields.
Furthermore, for a scalar function $\psi$ and a vector field
$\vv = (v_1, v_2)$ we set
\[
\ROT \psi \defeq \left(\frac{\partial \psi}{\partial y}\,, - \frac{\partial \psi}{\partial x}  \right)^\top \,,
\qquad
\rot \vv \defeq \frac{\partial v_2}{\partial x} - \frac{\partial v_1}{\partial y} \,.
\]
Finally we recall the following well known functional spaces which will be useful in the sequel
\begin{gather*}
    H(\diver, \omega) \defeq \{\bm{v} \in [L^2(\omega)]^2: \,   \diver
    \bm{v} \in L^2(\omega)\} \,,
    \\
    H(\rot, \omega) \defeq \{\bm{v} \in [L^2(\omega)]^2 :\,
    \rot \bm{v} \in L^2(\omega)\} \,.
\end{gather*}
}

\subsection{Mathematical model}
We consider a
(curved) domain $\Omega\subset\mathbb{R}^2$ with Lipschitz continuous boundary and external unit normal $\nn$. The boundary of $\Omega$ named
$\partial \Omega$ is divided into two parts $\partial_e \Omega$ and $\partial_n
\Omega$ such that $\overline{\partial \Omega} = \overline{\partial_e \Omega}
\cup \overline{\partial_n \Omega}$ and $\mathring{\partial_e \Omega} \cap
\mathring{\partial_n \Omega} = \emptyset$. For simplicity, we assume that
$\mathring{\partial_n \Omega} \neq \emptyset$.

For a positive definite tensor $\K$, a positive real number $\mu$, and a scalar source $f$, the following problem is set in $\Omega$:
\begin{problem}[Model problem]\label{pb:darcy_model}
    Find $(\bm{q}, p)$ such that
    \begin{subequations}
    \begin{align}\label{eq:darcy_model_eq}
        \left \{
        \begin{aligned}
            &\mu \bm{q} + \K \nabla p = \bm{0} \\
            &\diver \bm{q} + f = 0
        \end{aligned}
        \right .
        \qquad \text{in } \Omega,
    \end{align}
    supplied with the following boundary conditions
    \begin{equation}
    \label{eq:darcy_model}
    \left \{
    \begin{aligned}
        &p = \overline{p} & \qquad \text{on } \partial_n \Omega,\\
        &\bm{q} \cdot \bm{n} = \overline{q} & \qquad \text{on } \partial_e \Omega.
    \end{aligned}
    \right .
    \end{equation}
    \end{subequations}
\end{problem}
This problem describes, for example, the pressure $p$ and the Darcy velocity $\bm{q}$ of a single phase fluid in a porous medium, characterized by a permeability tensor $\K$, a fluid dynamic viscosity $\mu$, and fluid sinks/sources $f$.
%
{
In the following we assume null $\overline{q}$, otherwise a lifting technique should be considered.
Before introducing the weak problem associated to Problem \ref{pb:darcy_model},
we fix the following notation
\[
\VV \defeq \left\{ \bm{v} \in H(\diver, \Omega)
    \quad \text{s.t.} \quad \bm{v} \cdot \bm{n} = 0 \text{ on } \partial_e \Omega
    \right\}
\quad \text{and} \quad
Q \defeq L^2(\Omega) \,,
\]
equipped with natural inner products and induced norms.
}
The spaces $\VV$ and $Q$, with their structures, are thus Sobolev spaces.
In the previous definition of $\VV$ the condition on the essential part of
$\partial \Omega$ can be detailed as:
\begin{gather*}
    \langle \bm{v} \cdot \bm{n}, w \rangle = 0 \quad \forall w \in
    H^{\frac{1}{2}}_{00}(\partial_e \Omega)
\end{gather*}
where $\langle\cdot, \cdot\rangle$ is the duality pair from
$H^{-\frac{1}{2}}(\partial_e \Omega)$ to $H^{\frac{1}{2}}_{00}(\partial_e
\Omega)$. See \cite{Boffi2013} for more details.

We introduce now the weak formulation of Problem \ref{pb:darcy_model}.
The procedure is rather standard and leads to the definition of the following forms
\begin{align}\label{eq:continuous_forms}
    \begin{aligned}
        &a(\cdot, \cdot)\colon \VV \times \VV \to \mathbb{R} \qquad
        &&a(\bm{u}, \bm{v}) \defeq (\mu \K^{-1} \bm{u}, \bm{v})_\Omega \quad &\forall (\bm{u},
        \bm{v}) \in \VV \times \VV\\
        &b(\cdot, \cdot)\colon \VV \times Q \to \mathbb{R} \qquad
        &&b(\bm{u}, v) \defeq -(\diver \bm{u}, v)_\Omega \quad &\forall(\bm{u}, v)\in \VV
        \times Q.
    \end{aligned}
\end{align}
We have furthermore assumed that $\kappa \in
[L^\infty(\Omega)]^{2\times2}$, $\mu \in
L^\infty(\Omega)$ and it exists $\mu_0 > 0$ such that $\mu \geq \mu_0$.
Linear functionals associated to given data are defined as
\begin{align}\label{eq:linear_functionals}
    \begin{aligned}
    &G(\cdot)\colon \VV \to \mathbb{R} \quad
    &&G(\bm{v})\defeq  - (\overline{p}, \bm{v} \cdot \bm{n})_{\partial_n\Omega}
    \quad &\forall \bm{v} \in \VV\\
    &F(\cdot) \colon Q\to \mathbb{R} \quad
    &&F(s) \defeq (f, v)_\Omega \quad &\forall v \in Q,
    \end{aligned}
\end{align}
where the data have regularity $\overline{p} \in
H^{\frac{1}{2}}_{00}(\partial_n\Omega)$, and $f \in L^2(\Omega)$. We can finally
summarize the weak formulation of Problem \ref{pb:darcy_model} as the following.
\begin{problem}[Weak problem]\label{pb:darcy_weak}
    Find the couple Darcy velocity and pressure $(\bm{q}, p)~\in~\VV~\times~Q$ such that
    \begin{align}
       \left \{
        \begin{aligned}
            & a(\bm{q}, \bm{v}) + b(\bm{v}, p) = G(\bm{v}) && \forall \bm{v} \in \VV
            \\
            & b(\bm{q}, v) = F(v) && \forall v \in Q.
        \end{aligned}
         \right.
    \end{align}
\end{problem}
\noindent
The previous Problem is well posed (see for instance \cite{Boffi2013}).

{
\subsection{Assumptions on the curved domains}

Following the approach in \cite{BeiraodaVeiga2019}, we here detail the assumption
on the (curved) domain $\Omega$.  We consider a bounded Lipschitz domain $\Omega$
whose boundary $\partial \Omega$ is made up of a finite number of smooth curves
$\{\Gamma_i\}_{i=1, \dots, N}$ that fit the boundary split into  ``essential''
and  ``natural'' part, i.e.,
$$
\bigcup_{i=1}^{N_e} \Gamma_i = \partial_e \Omega\qquad\text{and}\qquad
\bigcup_{i=N_e+1}^{N} \hspace{-0.5em}\Gamma_i = \partial_n \Omega.
$$
We assume that:
\begin{assumption}[Boundary regularity]\label{ass:regu}
We assume that each curve $\Gamma_i$ of $\partial \Omega$ is sufficiently smooth, for instance we require that $\Gamma_i$ is of class $C^{m+1}$
with $m \geq 0$, i.e., there exists a given regular and invertible $C^{m+1}$-parametrization
$\gamma_i \colon I_i \to \Gamma_i$ for $i=1, \dots, N$, where  $I_i \defeq [a_i, b_i] \subset \R$ is a closed interval.
\end{assumption}
Since all the parts $\Gamma_i$ of $\partial \Omega$ will be treated in the same
way, in the following we will drop the index $i$ from all the involved maps and
parameters, in order to obtain a lighter notation.
}

\begin{remark}[Internal interfaces]
    It is important to note that proposed approach is also valid for internal curved interfaces.
    However, to keep the presentation simple we assume only curved elements on the boundary, being
    its extension straightforward. Examples in Subsection \ref{sub:inside1} and
    \ref{sub:inside2} deal with internal interfaces.
\end{remark}

{
\section{Mixed Virtual Elements on curved polygons}\label{sec:discrete_spaces}

In this section, we define the virtual formulation of Problem
\ref{pb:darcy_weak}.  We first discuss the assumptions for the meshes on the
curved domain $\Omega$, then we introduce the space for the vector and scalar
fields with the associated set of degrees of freedom. We discuss the
computability of the $L^2$-projection onto the polynomial space and define the
approximated linear form.

\subsection{Mesh assumptions}
From now on, we will denote with $E$ a general polygon having $\ell_e$ edges
$e$, which may any number of curved edges.
For each polygon $E$ and each edge $e$ of $E$ we denote
by $|E|$, $h_E$, $\xx_E=(x_E, y_E)$ the measure, diameter and centroid of $E$,
respectively.
By $h_e$, $\xx_e$  we denote the length and midpoint of $e$, respectively.
Furthermore, $\nn_E^e$ denotes the unit outward normal
vector to $e$ with respect to $E$, while
$\nn_E$ is a generic outward normal of $\partial E$.
We call $\nn^e$ a fixed unit normal vector which is normal to the edge $e$ and
$\sigma_{E,e} \defeq \nn_E^e \cdot \nn^e= \pm 1$ (notice that $\nn^e$ does not depend
on $E$).

Let $\Omega_h$ be a decomposition of $\Omega$ into general polygons $E$ completed along
$\partial \Omega$  by curved elements whose boundary contains an arc $\subset \partial \Omega$, where
we define $h~\defeq ~\sup_{E \in \Omega_h} h_E$, see~\cite{BeiraodaVeiga2019}. 
We make two assumptions on the mesh elements: there exists a positive uniform constant $\rho$ such that
\begin{assumption}[Star-shaped]\label{ass:star}
    Each element $E$ in $\Omega_h$ is star-shaped with respect to a
    ball $B_E$ of radius $ \geq\, \rho \, h_E$.
\end{assumption}
\begin{assumption}[Edges comparable size]\label{ass:mesh}
    For each element $E$ in $\Omega_h$, for any (possibly curved) edge
    $e$ of $E$, it holds $h_e \geq \rho \, h_E$.
\end{assumption}
We denote by $\Ed$ the set of all the mesh edges divided into internal
$\Edi$ and external $\Edb$ edges; the latter is split  into  ``essential edges'' $\Ede$  and  ``natural edges'' $\Edn$. For any $E \in \Omega_h$ we denote by $\EdE$ the set of the edges of $E$.
Finally the total number of edges (excluding the ``essential edges'' $\Ede$) and elements in the decomposition $\Omega_h$ are denoted by $L_e$ and $L_E$, respectively.

With a slight abuse of notation, we define the following maps to deal with both straight and curved edges:
\begin{itemize}
\item for any curved edge $e \in \Ed$, we call $\gamma \colon \Ie \subset I \to e$  the restriction of $\gamma \colon I \to \partial \Omega$ having image $e$,
\item for any straight edge $e \in \Ed$  with endpoints $\xx_{e_1}$ and $\xx_{e_2}$,
we denote by $\gamma \colon \Ie \defeq [0, h_e] \to e$ the standard affine map
$\gamma(t) = \frac{t}{h_e}(\xx_{e_2} - \xx_{e_1}) + \xx_{e_1}$.
\end{itemize}

\begin{remark} \label{rm:length}
We notice that, since the parametrization $\gamma \colon I \to \partial \Omega$
is fixed once and for all, under Assumption \ref{ass:regu}, it follows
that for any curved edge $e \in \EdE$, the length of the interval $\Ie$ is
comparable with the diameter $h_E$ of the element $E$,
since $h_e = \int_{\Ie} \|\gamma'(s)\| \, {\rm d}s$ and $\gamma$, $\gamma^{-1} \in W^{1, \infty}$ are fixed.
%
Moreover, since $\gamma$ is fixed, when $h$ approaches  zero  the straight segment $e'$
whose endpoints are vertexes of $e$ approaches the curved edge $e$.
Therefore by Assumption  \ref{ass:mesh}, for sufficiently small $h$, the length $h_e$ of the curved edge $e$ is comparable with the diameter $h_E$.
\end{remark}

In the following the symbol $\lesssim$  will denote a bound up to a generic positive
constant, independent of the mesh size $h$, but which may depend on $\Omega$, on the
``polynomial'' order $k$, on the parametrization $\gamma$ in Assumption
\ref{ass:regu} and on the shape constant $\rho$ in Assumptions \ref{ass:star}
and \ref{ass:mesh}.
}

{
\subsection{Polynomial and mapped polynomial spaces}

Using standard VEM notations, for $n \in \mathbb{N}$, $s \in \R^+$, and for any $E\in \Omega_h$, let us introduce the spaces:
\begin{itemize}
\item $\Pk_n(E)$ the set of polynomials on $E$ of degree $\leq n$  (with $\Pk_{-1}(E)=\{ 0 \}$),
\item $\Pk_n(\Omega_h) \defeq \{q \in L^2(\Omega_h): q_{|E} \in
\Pk_n(E) \, \forall E \in \Omega_h\}$,
\item $H^s(\Omega_h) \defeq \{v \in L^2(\Omega_h): v_{|E} \in  H^s(E)\, \forall E \in \Omega_h\}$ equipped with the broken norm and seminorm
\[
\|v\|^2_{s,\Omega_h} \defeq \sum_{E \in \Omega_h} \|v\|^2_{s,E}\,,
\qquad
|v|^2_{s,\Omega_h} \defeq \sum_{E \in \Omega_h} |v|^2_{s,E} \,,
\]
\end{itemize}
and we define
\[
\pi_n\defeq \dim(\Pk_n(E)) = \frac{(n+1)(n+2)}{2} \,.
\]
Notice that the following useful polynomial decomposition holds \cite{BeiraodaVeiga2014b,Dassi2019}
\begin{equation}
\label{eq:polydec}
[\Pk_n(E)]^2 = \nabla \Pk_{n+1}(E) \oplus \xxp \Pk_{n-1}(E)
\end{equation}
where $\xxp\defeq (y, -x)^T$.

\begin{remark} \label{rm:rot}
Note that \eqref{eq:polydec} implies that the operator $\rot$ is an isomorphism from $\xxp \Pk_{n-1}(E)$
to the whole $\Pk_{n-1}(E)$, i.e., for any $q_{n-1} \in \Pk_{n-1}(E)$ there exists a unique $p_{n-1} \in \Pk_{n-1}(E)$ such that $q_{n-1} = \rot(\xxp p_{n-1})$.
\end{remark}

A natural basis associated with the space $\Pk_n(E)$ is the set of normalized monomials
\[
\Mk_n(E) \defeq \left\{ \left( \frac{\xx - \xx_E}{h_E} \right)^{\bm{\beta}} \text{ with } \abs{\bm{\beta}} \leq n \right\}
\]
where $\bm{\beta}$ is a multi-index. Notice that $\|m\|_{L^{\infty}(E)} \leq 1$ for any $m \in \Mk_n(E)$.
We extend the basis $\Mk_n(E)$ for vector valued polynomials $[\Pk_n(E)]^2$ defining
\[
\MMk_n(E) \defeq \left\{ (m_r, 0)^\top\,, \,\, (0, m_s)^\top\, \quad \text{with $m_r, m_s \in \Mk_n(E)$}   \right\}.
\]

Let us now introduce the boundary space on the edge $e \in \Ed$.
Following the same approach, for any interval $\Ie \subset \R$ we denote by
$\Pk_n(\Ie)$ the set of polynomials on $\Ie$ of degree $\leq n$ with the
associated basis of normalized polynomials
\[
    \Mk_n(\Ie) \defeq \left\{ 1, \frac{x-x_{\Ie}}{h_{\Ie}},\left(
    \frac{x - x_{\Ie}}{h_{\Ie}}\right)^2, \ldots, \left(
    \frac{x - x_{\Ie}}{h_{\Ie}}\right)^n \right\},
\]
again we notice that $\|m\|_{L^{\infty}(\Ie)} \leq 1$ for any $m \in \Mk_n(\Ie)$.
For each  edge $e \in \mathcal{E}_h$ we consider the following mapped
polynomial and scaled monomial spaces
\begin{eqnarray*}
    \Pkt_n(e) &\defeq& \{ \widetilde{q} = q \circ \gamma^{-1}: q \in
    \Pk_n(\mathfrak{e}) \}\quad\text{and}\quad\\
    \Mkt_n(e) &\defeq& \{ \widetilde{m} = m \circ \gamma^{-1}: m \in
    \Mk_n(\mathfrak{e}) \}\,,
\end{eqnarray*}
i.e., $\Pkt_n(e)$ is made of all functions that are polynomials with respect to the parametrization $\gamma$.
It is important to note that   the following property  holds:
\begin{property}\label{prop:subset}
For any edge $e \in \EdE$ we have
$\Pk_n(E)|_e \subset \Pkt_n(e)$ if $e$ is straight, or
$\Pk_0(E)|_e \subset \Pkt_n(e)$ and
$\Pk_i(E)|_e \not\subset \Pkt_n(e)$, for $i > 0$, if $e$ is curved.
The same considerations apply to $\Mkt_n$.
\end{property}

Finally the local $L^2$-projection operator $\Pi^n_0 \colon [L^2(E)]^2 \to
[\Pk_n(E)]^2$ is defined as follows: given $\ww \in [L^2(E)]^2$ we have
\begin{equation}
\label{eq:projection_def}
    \int_E \Pi^n_0 \ww \cdot \bm{m} \, {\rm d}E =
    \int_E \ww \cdot \bm{m} {\rm d}E \quad\forall
    \bm{m} \in \MMk_n(E).
\end{equation}
With a slight abuse of notation, we denote by $\Pi^n_0 \colon [L^2(\Omega)]^2
\to [\Pk_n(\Omega_h)]^2$ the projection onto the space of piecewise polynomials
defined element-wise by $(\Pi^n_0 \ww)|_E \defeq \Pi^n_0 (\ww|_E)$ for all $E\in
\Omega_h$.
Similarly the $L^2$-edge projection operator $\widetilde{\Pi}^n_0 \colon L^2(e)
\to \Pkt_n(e)$ is defined as follows: given $w \in L^2(e)$
\begin{equation}
\label{eq:projection_edge}
    \int_e \widetilde{\Pi}^n_0 w \, \widetilde{m} \, {\rm d}e =
    \int_e  w \, \widetilde{m} \, {\rm d}e \quad\forall
    \widetilde{m} \in \Mkt_n(e).
\end{equation}
}

{
\subsection{Vector space}\label{subsec:vector_spaces}

Let $k \geq 0$ be the polynomial degree of accuracy of the method.
We proceed as in a standard virtual element fashion, i.e., 
we firstly define the virtual spaces element-wise then 
we globally glue them.
We introduce the local virtual space on the curved element $E\in \Omega_h$:
\begin{eqnarray}
\VV_k(E) \defeq \{\vv \in  H(\diver, E) \cap H(\rot, E)&:&
\vv \cdot\nn^e \in \Pkt_k(e) \, \forall e \in \EdE,\nonumber\\
\phantom{\VV_k(E) \defeq \{\vv \in  H(\diver, E) \cap H(\rot, E)}&\phantom{:}&
\diver \vv \in \Pk_k(E),\, \rot \vv \in \Pk_{k-1}(E) \}\,.\nonumber\\
\label{eqn:spaceCurved}
\end{eqnarray}
The definition above extends to the curved elements the ``straight'' mixed VEM space introduced in
\cite{BeiraodaVeiga2014b,BeiraoVeiga2016} that is the VEM counterpart of the Raviart-Thomas spaces to more general element geometries.
An element $\vv$ belonging to the space $\VV_k(E)$ is well defined (assuming the compatibility condition of the divergence Theorem), but it is not
a-priori specified in the internal part of $E$ as done in the standard finite
elements.

We have the following choice for the degrees of freedom.
\begin{dof}[DoFs for $\VV_k(E)$]\label{dof:vhk}
    The set of scaled degrees of freedom associated to the space $\VV_k(E)$ are
    given for all $\ww \in \VV_k(E)$, by the linear operators $\boldsymbol{D}$ split into three subsets:
    \begin{itemize}
    \item $\boldsymbol{D_1}$: the boundary moments
    \[
    \boldsymbol{D_1}^{e,i}(\ww) \defeq \frac{1}{h_e} \int_e \ww \cdot \nn^e
    \widetilde{m}_i \, {\rm d}e
    \quad \forall e \in \EdE, \, \forall \widetilde{m}_i \in \Mkt_k(e), i=1, \dots, k+1;
    \]
    \item $\boldsymbol{D_2}$: the element moments of the divergence
    \[
    \boldsymbol{D_2}^{j}(\ww) \defeq
    \frac{h_E}{|E|} \int_E \diver \ww \, m_j \, {\rm d}E
    \quad \forall m_j \in \Mk_k(E)\setminus \Mk_0(E),\, j=2, \dots, \pi_k;
    \]
    \item $\boldsymbol{D_3}$: the element moments
    \[
    \boldsymbol{D_3}^l(\ww) \defeq
    \frac{1}{|E|} \int_E \ww \cdot \mmp m_l \, {\rm d}E
    \quad \forall m_l \in \Mk_{k-1}(E),\, l=1, \dots, \pi_{k-1},
    \]
    where $\mmp:= \left(\mathlarger{\frac{(y - y_E)}{h_E}},\,-\mathlarger{\frac{(x - x_E)}{h_E}}\right)$.
    \end{itemize}
\end{dof}
\noindent
The dimension of $\VV_k(E)$ is given by
\begin{equation}
\label{eq:diml}
\dim(\VV_k(E)) = \ell_e (k+1) + (\pi_{k} - 1) + \pi_{k-1} \,.
\end{equation}

\begin{remark}
\label{rm:dofs1}
The proof that the linear operators  $\boldsymbol{D_1}$,  $\boldsymbol{D_2}$ and
$\boldsymbol{D_3}$ constitute a set of DoFs for $\VV_k(E)$ follows the same
guidelines of Lemma 3.1, Lemma 3.2 and Theorem 3.1 in \cite{BeiraodaVeiga2014b}.
\end{remark}

\begin{remark}
\label{rm:dofs2}
The set of DoFs $\boldsymbol{D_2}$ used in the present work is different from the one suggested in \cite{BeiraodaVeiga2014b}, where, instead, the moments
\[
    \frac{h_E}{|E|} \int_E \ww \cdot \nabla m_{k} \, {\rm d}E
    \quad \forall m_{k} \in \Mk_{k}(E)\setminus \Mk_0(E).
\]
are used. The choice proposed in the present work turns out to be particularly suited for curved elements as explained in Remark \ref{rm:dofs3}.
\end{remark}

The global space is defined by gluing together all local spaces, which is thus set as
\begin{equation}
\label{eq:VVG}
    \VV_k(\Omega_h) \defeq \{ \vv \in \VV: \, \vv|_E \in
    \VV_k(E)\, \forall E \in \Omega_h \}.
\end{equation}
More specifically, we require that for any internal edge $e \in \EdE \cap \mathcal{E}_h^{E'}$
\[
\vv|_E \cdot \nn_e^E + \vv|_{E^\prime} \cdot \nn_e^{E^\prime} = 0
\quad \forall \vv \in \VV_k(\Omega_h),
\]
that is in accordance with the DoFs definition $\boldsymbol{D_1}$.
The dimension of $\VV_k(\Omega_h)$ is thus given by
\begin{equation}
\label{eq:dimg}
\dim(\VV_k(\Omega_h)) = L_e (k+1) + (\pi_{k} - 1)L_E + \pi_{k-1} L_E \,,
\end{equation}
where $L_e$ and $L_E$ are the number of edges and polygons in $\Omega_h$, respectively.

\subsection{Scalar space}\label{subsec:scalar_spaces}

The approximation of the continuous space $Q$ is made of piecewise discontinuous
polynomials in each element. The space $Q_k(\Omega_h) \subset Q$ belongs to the
standard finite elements and its elements can be easily handled. Namely for $k
\geq 0$, we have
\begin{gather*}
    Q_k(E)\defeq \{ v \in L^2(E):\, v \in \mathbb{P}_k(E) \}.
\end{gather*}
For this space we consider the following DoFs
\begin{dof}[DoFs for $Q_k(E)$]\label{dof:qhk}
    The internal scaled moments are the DoFs for $Q_k(E)$, i.e., for any $v \in Q_k(E)$ we consider
    \begin{itemize}
    \item $\boldsymbol{D_Q}$: the element moments
        \[
        \boldsymbol{D_Q}^r(v)\defeq
           \frac{1}{|E|} \int_E v \,m_r \, {\rm d}E \qquad  \forall m_r \in
           \Mk_k(E), \, r=1, \dots, \pi_k.
        \]
    \end{itemize}
\end{dof}
\noindent
We define the global discrete space as
\begin{equation}
\label{eq:qh}
    Q_k(\Omega_h)\defeq \{v \in Q:\, v|_E \in Q_k(E)\}.
\end{equation}
Notice that by construction
we have
$\diver (\VV_k(\Omega_h)) = Q_k(\Omega_h)$.
}

{
\subsection{Polynomial projector and discrete forms}\label{sec:discrete_forms}

As for the straight virtual spaces, a function $\ww \in \VV_k(E)$ is not known in closed form, however exploiting the DoFs values of $\ww$ we can compute some fundamental informations.
\paragraph{The polynomial $\ww \cdot \nn^e$ is computable}
We start by noticing that the normal component $\ww \cdot \nn^e$ is explicitly known for all $e \in \EdE$.
Indeed, being $\ww \cdot\nn^e \in \Pkt_k(e)$, there exist $c_1, \dots, c_{k+1} \in \R$ such that
\begin{equation}
\label{eq:wwb}
    \ww \cdot \nn^e \!=\!
    \sum_{\rho = 1}^{k+1} c_\rho \widetilde{m}_\rho  \!=\!
    \sum_{\rho = 1}^{k+1} c_\rho {m}_\rho \circ \gamma^{-1}
    \quad \text{with }\widetilde{m}_\rho \in \Mkt_k(e)\text{ and }{m}_\rho \in \Mk_k(\Ie).
\end{equation}
In order to compute the coefficients $c_\rho$ we exploit the DoFs $\boldsymbol{D_1}$:
\begin{gather*}
\boldsymbol{D_1}^{e,i}(\ww) =
\frac{1}{h_e}\int_e \ww \cdot \nn^e \, \widetilde{m}_i \, {\rm d}e  =
\sum_{\rho = 1}^{k+1}\frac{c_\rho}{h_e} \int_e \widetilde{m}_\rho \,
\widetilde{m}_i \, {\rm d}e =
\sum_{\rho = 1}^{k+1} \frac{c_\rho}{h_e} \int_{\Ie} m_\rho \, m_i \, \|\gamma^\prime\| {\rm d}t
\end{gather*}
for $i=1, \dots, k+1$. Then it is possible to compute the coefficients $c_\rho$ and thus the explicit expression of $\ww \cdot \nn^e$ for any edge $e \in \EdE$.

\paragraph{The polynomial $\diver \ww$ is computable}
In such framework we can explicitly compute $\diver \ww$ via $\boldsymbol{D_1}$ and $\boldsymbol{D_2}$.
Indeed, being $\diver \ww \in \Pk_k(E)$, there exist $d_1, \dots, d_{\pi_k} \in \R$ such that
\begin{equation}
\label{eq:wwd}
    \diver \ww =
    \sum_{\theta = 1}^{\pi_k} d_\theta m_\theta
    \quad \text{with ${m}_\theta \in \Mk_k(E)$,}
\end{equation}
then it follows that
\[
\frac{h_E}{|E|}\int_E \diver \ww \, m_j \, {\rm d}E= \frac{h_E}{|E|}
\sum_{\theta = 1}^{\pi_k} d_\theta \int_E m_\theta m_j \, {\rm d}E
\quad \forall m_j \in \Mk_k(E),\, j=1, \dots, \pi_k.
\]
As before the right-hand side matrix is computable, whereas the left-hand side
corresponds to the DoFs $\boldsymbol{D_2}^j(\ww)$ if $m_j \in \Mk_k(E) \setminus
\Mk_0(E)$, for $j=1$ we exploit the boundary information:
\[
\frac{h_E}{|E|}\int_E \diver \ww  \, {\rm d}E=
\frac{h_E}{|E|}\int_{\partial E} \ww \cdot \nn_E \, {\rm d}e
= \sum_{e \in \EdE} \sigma_{E,e} \frac{h_E h_e}{|E|} \frac{1}{h_e} \int_e \ww
\cdot \nn^e  \, {\rm d}e
\]
that, recalling Property \ref{prop:subset}, is a linear combination of DoFs $\boldsymbol{D_1}^{e,1}(\ww)$.

\paragraph{The projection $\Pi^k_0$ is computable}
The computations above allow us to evaluate the projection  $\Pi^k_0 \ww$ for all $\ww \in \VV_k(E)$.
We consider first the following expansion on vector monomials
\begin{gather*}
    \Pi_0^k \ww = \sum_{\xi = 1}^{2\pi_k} w_\xi \bm{m}_\xi
    \qquad \text{with  $\bm{m}_\xi \in \MMk_k(E)$}
\end{gather*}
and then we use definition \eqref{eq:projection_def} to obtain
\begin{gather*}
    \int_E \ww \cdot \bm{m}_s \, {\rm d}E =
    \int_E \Pi^k_0 \ww \cdot \bm{m}_s \, {\rm d}E =
    \sum_{\xi = 1}^{2\pi_k} w_\xi \int_E \bm{m}_\xi \cdot \bm{m}_s \, {\rm d}E
\end{gather*}
for all $\bm{m}_s \in \MMk_k(E)$, $s=1, \dots, 2\pi_k$.
Unfortunately, the first term involves a virtual function $\bm{w}$ which makes it
not computable as it is.
To proceed, we can use the decomposition \eqref{eq:polydec} of
$\bm{m}_s$ obtaining
\[
\bm{m}_s = \nabla p_{k+1} + \sum_{l=1}^{\pi_{k-1}} g_l \mmp \, m_l
\]
for a suitable polynomial $p_{k+1} \in \Pk_{k+1}(E)  \setminus \Pk_0(E)$
and suitable coefficients $g_1, \dots, g_{\pi_{k-1}} \in \R$.
Therefore integrating by parts, \eqref{eq:wwb} and \eqref{eq:wwd} yield
\[
\begin{aligned}
&\int_E \bm{w} \cdot \bm{m}_s \, {\rm d}E =
\int_E \bm{w} \cdot \nabla p_{k+1} \, {\rm d}E  +
\sum_{l=1}^{\pi_{k-1}} g_l \int_E \bm{w} \cdot \mmp \, m_l \, {\rm d}E
\\
&=
\int_{\partial E} \bm{w} \cdot \nn_E  p_{k+1} \, {\rm d}e  -
\int_E \diver \bm{w} \, p_{k+1} \, {\rm d}E  +
\sum_{l=1}^{\pi_{k-1}} g_l \int_E \bm{w} \cdot \mmp \, m_l \, {\rm d}E
\\
&=
\sum_{e \in \EdE} \sigma_{E,e} \sum_{\rho=1}^{k+1} c_{\rho} \int_{e}
\widetilde{m}_\rho p_{k+1} \, {\rm d}e  -
\sum_{\theta=1}^{\pi_k} d_\theta \int_E m_\theta  p_{k+1} \, {\rm d}E  +
|E|\sum_{l=1}^{\pi_{k-1}} g_l  \boldsymbol{D_3}^l(\ww)
\end{aligned}
\]
that is a computable expression.

Following a standard procedure, we define the computable discrete local form
$a_k^E(\cdot, \cdot) \colon \VVb_k(E) \times \VVb_k(E) \to \R$, with $\VVb_k(E)\defeq \VV_k(E) + [\Pk_k(E)]^2$, given by
\begin{equation}
\label{eq:ahE}
a_k^E(\uu_h, \vv_h) \defeq \int_E \mu \kappa^{-1} \Pi_0^k \uu_h \cdot \Pi_0^k \vv_h \, {\rm d}E +
\nu(E) \St((I - \Pi_0^k)\uu_h, (I - \Pi_0^k)\vv_h)
\end{equation}
for all $\uu_h,\vv_h \in \VVb_k(E)$. In the previous definition
the term $\nu(E) \in
\R$ is a cell-wise approximation of the physical parameters $\mu k^{-1}$ and the
stabilization form $\St(\cdot, \cdot) \colon \VVb_k(E) \times \VVb_k(E)  \to \R$ is defined by
\[
\St(\uu_h, \vv_h) \defeq   \abs{E} \sum_{s = 1}^{N_{\dofop}(E)} \boldsymbol{D}^s(\uu_h) \boldsymbol{D}^s(\vv_h)
\]
that is
\begin{multline}
\label{eq:St}
\St(\uu_h, \vv_h) \defeq
|E| \sum_{e \in \EdE} \sum_{i=1}^{k+1} \boldsymbol{D_1}^{e,i}(\uu_h) \boldsymbol{D_1}^{e,i}(\vv_h) +
\\+
|E|      \sum_{j=2}^{\pi_k} \boldsymbol{D_2}^{j}(\uu_h) \boldsymbol{D_2}^{j}(\vv_h) +
|E|    \sum_{l=1}^{\pi_{k-1}} \boldsymbol{D_3}^{l}(\uu_h) \boldsymbol{D_3}^{l}(\vv_h)
\end{multline}
for all $\uu_h,\vv_h \in \VVb_k(E)$, being $N_{\dofop}(E)$ the total number of DoFs on $E$.
Since the global form is the sum of the local counterparts, we obtain $a_k(\cdot, \cdot) \colon
\VVb_k(\Omega_h)\times \VVb_k(\Omega_h) \to \R$ defined by
\begin{equation}
\label{eq:ah}
        a_k(\uu_h, \vv_h) \defeq \sum_{E \in \Omega_h} a_k^E(\uu_h, \vv_h)
        \quad
        \forall \uu_h, \vv_h \in \VVb_k(\Omega_h)
\end{equation}
\begin{remark}[On the space $\VVb_k$]
    In the definition of the local discrete form $a_k^E$ \eqref{eq:ahE}, we have
    considered the sum space $\VVb_k(E)$ for both of its entries. In fact, as
    reported in Property \ref{prop:subset}, the space $\VV_k(E)$ may not contain
    all the polynomials up to degree $k$. However, in order to have the optimal rate of convergence for
    the proposed scheme, we need to verify the continuity of $a_k^E$ on the sum space $\VVb_k(E)$
    (cfr. Proposition \ref{pr:continuity}).
\end{remark}

\subsection{The discrete problem}
\label{sub:dp}

Referring to the discrete spaces \eqref{eq:VVG} and \eqref{eq:qh}, the discrete form \eqref{eq:ah}, the virtual element approximation
of the Darcy equation is given by
\begin{problem}[VEM problem]\label{pb:darcy_vem}
    Find the couple Darcy velocity and pressure $(\bm{q}_h, p_h) \in \VV_k(\Omega_h) \times Q_k(\Omega_h)$ such that
    \begin{align}
       \left \{
        \begin{aligned}
            & a_k(\bm{q}_h, \bm{v}_h) + b(\bm{v}_h, p_h) = G(\bm{v}_h) &&
            \forall \bm{v}_h \in \VV_k(\Omega_h)
            \\
            & b(\bm{q}_h, v_h) = F(v_h) && \forall v_h \in Q_k(\Omega_h).
        \end{aligned}
         \right.
    \end{align}
\end{problem}
\noindent
Notice that since for any function $\vv_h \in \VV_k(\Omega_h)$ its divergence and its boundary values are explicitly known, we do not need to introduce any approximation for the form $b(\cdot, \cdot)$ and for the linear function $G(\cdot)$.
}

{
\section{Theoretical analysis} \label{sec:theory}

In this section, we introduce an interpolation operator that allows us to show
the inf-sup stability of the proposed scheme. After, the stability of the
stabilization term is studied.

\subsection{Interpolation and Inf-sup stability}
\label{sub:int}

We start by reviewing a classical approximation result for polynomials on star-shaped domains, see for instance \cite{brenner-scott:book}.

\begin{lemma}[Bramble-Hilbert]
\label{lm:bramble}
Under Assumption \ref{ass:star}, let $0 \leq s \leq k+1$.
Then, referring to \eqref{eq:projection_def}, for all $\vv \in \VV \cap H^s(\Omega_h)$ it holds
\[
\|\vv - \Pi^k_0 \vv\|_{\Omega_h,0} \lesssim h^s \, |\vv|_{\Omega_h,s} \,.
\]
\end{lemma}
Let us introduce the linear Fortin operator $\Pi^k_{\rm F} \colon [H^1(\Omega)]^2 \to \VV_k(\Omega_h)$
defined through the DoFs $\boldsymbol{D_1}$,  $\boldsymbol{D_2}$ and
$\boldsymbol{D_3}$. For $\ww \in [H^1(\Omega)]^2$ and for all $e \in \Ed$ and $E
\in \Omega_h$,  we require the following three conditions
\begin{eqnarray}
\label{eq:fortin1}
\hspace{-1em}\int_e (\ww - \Pi^k_{\rm F} \ww) \cdot \nn^e \widetilde{m}_i \, {\rm
d}e = 0
&\forall&\hspace{-1em}\widetilde{m}_i \in \Mkt_k(e),\, i=1, \dots, k+1;
\\
\label{eq:fortin2}
\hspace{-1em}\int_E \diver (\ww - \Pi^k_{\rm F} \ww) \, m_j \, {\rm d}E = 0
&\forall&\hspace{-1em} m_j \in \Mk_k(E)\setminus \Mk_0(E),\, j=2, \dots, \pi_k;
\\
\label{eq:fortin3}
\hspace{-1em}\int_E (\ww -\Pi^k_{\rm F} \ww) \cdot \mmp m_l \, {\rm d}E = 0
&\forall&\hspace{-1em} m_l \in \Mk_{k-1}(E),\, l=1, \dots, \pi_{k-1}.
\end{eqnarray}
The definition above easily implies that the following diagram
\begin{equation}
\label{eq:diagram}
\begin{split}
[H^1(\Omega)]^2 \,
&\xrightarrow[]{  \,\,\,\,\, \text{{$\diver$}} \,\,\,\,\,  } \quad \, \, \,
Q \quad \, \, \,
\xrightarrow[]{\, \, \,\,\,\,\, 0 \,\,\,\,\, \, \,}\,
0
\\
\Pi^k_{\rm F}  \bigg\downarrow \qquad &
\quad \quad \qquad
\Pi_0^k        \bigg\downarrow
\\
\VV_k(\Omega_h) \,
&\xrightarrow[]{  \,\,\,\,\, \text{{$\diver$}} \,\,\,\,\,  }\,
Q_k(\Omega_h) \, \,
\xrightarrow[]{\, \, \,\,\,\,\, 0 \,\,\,\,\, \, \,}\,
0
\end{split}
\end{equation}
where $0$  is the mapping that to every function associates the number $0$, is a commutative map. In particular, we have the following property:
\begin{equation}
\label{eq:divpf}
\diver (\Pi^k_{\rm F} \ww) = \Pi_0^k \diver \ww \quad \forall \ww \in [H^1(\Omega)]^2.
\end{equation}
Indeed, since $\diver (\Pi^k_{\rm F} \ww) \in Q_k(\Omega_h)$, by definition of $\Pi_0^k$, we need to verify that for all $E \in \Omega_h$
\[
\int_E \diver (\ww - \Pi^k_{\rm F} \ww)\,m_j \, {\rm d}E= 0
\quad \forall m_j \in \Mk_k(E),\, j=1, \dots, \pi_k.
\]
If $m_j \in \Mk_k(E) \setminus \Mk_0(E)$ it follows by  \eqref{eq:fortin2},
whereas if $j=1$ by Property \ref{prop:subset}  and \eqref{eq:fortin1} we have
\[
\begin{aligned}
\int_E \diver (\ww - \Pi^k_{\rm F} \ww)\, {\rm d}E &=
\int_{\partial E} (\ww - \Pi^k_{\rm F} \ww) \cdot \nn_E \, {\rm d}e \\
&=\sum_{e \in \EdE} \sigma_{E,e} \int_e (\ww - \Pi^k_{\rm F} \ww) \cdot \nn^e \,
{\rm d}e = 0.
\end{aligned}
\]

\begin{remark}
\label{rm:dofs3}
Notice that property \eqref{eq:divpf} is strictly related to the DoFs $\boldsymbol{D_2}$ and the associated Fortin operator.
With the choice of DoFs of Remark \ref{rm:dofs2} and adopted for the
``straight'' MVEM \cite{BeiraodaVeiga2014b} with the associated Fortin operator we have instead
\begin{multline*}
\int_E \diver (\ww - \Pi^k_{\rm F} \ww)\,m_k \, {\rm d}E =
- \int_E (\ww - \Pi^k_{\rm F} \ww) \cdot \nabla m_k \, {\rm d}E + \\
+ \sum_{e \in \Ed} \sigma_{E, e} \int_e  (\ww - \Pi^k_{\rm F} \ww) \cdot \nn^e
m_k\, {\rm d}e \,.
\end{multline*}
For a curved polygon $E$, the second term is not zero any more since, as observed in
Property \ref{prop:subset}, the restriction of $m_k$ on a curved edge $e$ does not belong to $\Pkt_k(e)$.
Therefore the choice of $\boldsymbol{D_2}$ is particularly suited for curved polygons.
\end{remark}

\noindent
As a consequence of the above arguments we have the following results:
the first one deals with the approximation property of the space \eqref{eq:VVG} and follows
combining \eqref{eq:divpf} and Lemma \ref{lm:bramble} with \cite{duran},
the second one is associated with the commutativity of the diagram \eqref{eq:diagram} and deals with the inf-sup stability of the method \cite{Boffi2013}.
\begin{prop}
\label{pr:interpolation}
Let $\ww \in \VV \cap [H^{k+1}(\Omega_h)]^2$ with $\diver \ww \in H^{k+1}(\Omega_h)$ and let $\Pi^k_{\rm F}$ be the linear Fortin operator.
Then under Assumption \ref{ass:star} it holds
\[
\begin{gathered}
\|\ww - \Pi^k_{\rm F} \ww\|_{0, \Omega} \lesssim h^{k+1} \, |\ww|_{k+1, \Omega_h}\,,
\\
\|\diver \ww - \diver \Pi^k_{\rm F} \ww\|_{0, \Omega} \lesssim h^{k+1} \, |\diver \ww|_{k+1, \Omega_h}\,.
\end{gathered}
\]
\end{prop}

\begin{prop}
\label{pr:infsup}
Under Assumption \ref{ass:star} there exists $\beta >0$ such that
\[
\inf_{v \in Q_k(\Omega_h)} \sup_{\ww \in \VV_k(\Omega_h)}
\frac{b(\ww, v)}{\|v\|_{Q} \|\ww\|_{\VV}} \geq \beta \,.
\]
\end{prop}

\subsection{Stability analysis}
\label{sub:stab}

The aim of the section is to prove the stability bounds for the approximated
bilinear form \eqref{eq:ahE} and in particular for the stabilization term $\St$.
We want to prove that
\begin{gather*}
\St(\ww, \ww) \gtrsim \|\ww\|^2_{0,E} \quad \forall \ww \in \VV_k(E),
\\
\St(\ww, \ww) \lesssim \|\ww\|^2_{0,E} \quad \forall \ww \in
\VVb_k(E).
\end{gather*}
We start with the following useful inverse estimates.
\begin{lemma}
\label{lm:inverse}
We assume (\ref{ass:star}) and we fix an integer $n \in \mathbb{N}$.
Let $\ww \in H(\diver, E)$ such that $\diver \ww \in \Pk_n(E)$ then
\begin{equation}
\label{eq:divinv}
\|\diver \ww\|_{0,E} \lesssim h_E^{-1} \, \|\ww\|_{0,E} \,.
\end{equation}
Let $\ww \in H(\rot, E)$ such that $\rot \ww \in \Pk_n(E)$ then
\begin{equation}
\label{eq:rotinv}
\|\rot \ww\|_{0,E} \lesssim h_E^{-1} \, \|\ww\|_{0,E} \,.
\end{equation}
\end{lemma}
\begin{proof}
Under Assumption \ref{ass:star}, let $T_E \subset E$ be an equilateral triangle inscribed in the ball $B_E$. Then for any $p_n \in \Pk_n(E)$ it holds $\|p_n\|_{0,E} \lesssim \|p_n\|_{0, T_E}$.
Let $b_3 \in \Pk_3(T_E)$ be the cubic bubble with $\|b_3\|_{L^{\infty}(T_E)}=1$.
Then, applying a polynomial inverse estimate on $T_E$ we get
\[
\begin{split}
\|\diver \ww\|_{0,E}^2 & \lesssim \|\diver \ww\|_{0,T_E}^2\\[1.em]
&\lesssim \int_{T_E} b_3 \diver \ww \, \diver \ww \, {\rm d}E
= - \int_{T_E} \nabla(b_3 \diver \ww)  \ww \, {\rm d}E
\\[1.em]
& \lesssim \|\nabla (b_3 \diver \ww)\|_{0,T_E} \|\ww\|_{0,T_E}
\lesssim h_E^{-1} \|b_3 \diver \ww\|_{0,T_E} \|\ww\|_{0,T_E}
\\[1.em]
& \lesssim h_E^{-1} \|\diver \ww\|_{0,T_E} \|\ww\|_{0,T_E}
\lesssim h_E^{-1} \|\diver \ww\|_{0,E} \|\ww\|_{0,E} \,,
\end{split}
\]
from which follows \eqref{eq:divinv}.
The same argument applies to \eqref{eq:rotinv}.
\end{proof}

\begin{prop}
\label{pr:continuity}
Let $E \in \Omega_h$. Under Assumptions \ref{ass:regu}, \ref{ass:star} and~\ref{ass:mesh} the following holds
\[
\St(\ww, \ww) \lesssim \|\ww\|_{0,E}^2 \quad \forall \ww \in \VVb_k(E).
\]
\end{prop}
\begin{proof}
By definition \eqref{eq:St}, for $\ww \in \VVb_k(E)$ we need to
prove that
\begin{equation}
\label{eq:st1}
\St(\ww, \ww) = \sum_{e \in \EdE} \sum_{i=1}^{k+1} |E|\boldsymbol{D_1}^{e,i}(\ww)^2
+ \sum_{j=2}^{\pi_k} |E|\boldsymbol{D_2}^{j}(\ww)^2
+ \sum_{l=1}^{\pi_{k -1}} |E|\boldsymbol{D_3}^{l}(\ww)^2  \lesssim \|\ww\|_{0,E}^2 \,.
\end{equation}
We start analysing the first term in the left-hand side.
Employing the $H(\diver)$ trace inequality \cite[Theorem 3.24]{monk:book} and Lemma \ref{lm:inverse}  it holds
\[
h_E^{-1} \|\ww \cdot \nn^E\|_{0,\partial E}^2 \lesssim
h_E^{-2} \|\ww \|_{0,\partial E}^2 +\|\diver \ww \|_{0, E}^2 \lesssim h_E^{-2} \|\ww \|_{0,E}^2
\qquad \forall \ww \in \VVb_k(E).
\]
Then, since $\|m_i\|_{L^{\infty}(\Ie)}\leq 1$ and $h_{\Ie} \lesssim h_E$ (cfr. Remark \ref{rm:length}), it follows that
\begin{gather}
\label{eq:st2}
\begin{aligned}
\sum_{e \in \EdE} \sum_{i=1}^{k+1} &|E|\boldsymbol{D_1}^{e,i}(\ww)^2  =
\sum_{e \in \EdE} \sum_{i=1}^{k+1} \frac{|E|}{h_e^2} \left(\int_e \ww \cdot
\nn^e \widetilde{m}_i \, {\rm d}e \right)^2
\\
&\lesssim
\sum_{e \in \EdE} \sum_{i=1}^{k+1}  \|\ww \cdot \nn^e\|_{0,e}^2 \|\widetilde{m}_i\|_{0,e}^2
\lesssim
\sum_{e \in \EdE} \|\ww \cdot \nn^e\|_{0,e}^2  \sum_{i=1}^{k+1} \int_e
\widetilde{m}_i^2 \, {\rm d}e
\\
&\lesssim
\sum_{e \in \EdE} \|\ww \cdot \nn^e\|_{0,e}^2  \sum_{i=1}^{k+1} \int_{\Ie} m_i^2
\|\gamma'\| \, {\rm d}\Ie
\lesssim
\sum_{e \in \EdE} \|\ww \cdot \nn^e\|_{0,e}^2 h_E \lesssim \|\ww\|_{0,E}^2.
\end{aligned}
\end{gather}
Consider the second term of~\eqref{eq:st1},
we apply Lemma~\ref{lm:inverse} and,
since $\|m_i\|_{L^{\infty}(E)}\leq~1$,
we infer
\begin{gather}
\label{eq:st3}
\begin{aligned}
\sum_{j=2}^{\pi_k} |E|\boldsymbol{D_2}^{j}(\ww)^2 &=
\sum_{j=2}^{\pi_k} |E| \left(\frac{h_E}{|E|} \int_E \diver \ww \, m_j \, {\rm
d}E\right)^2 \\
&\lesssim
\frac{h_E^2}{|E|} \sum_{j=2}^{\pi_k} \|\diver \ww\|_{0,E}^2 \|m_j\|_{0,E}^2
\lesssim
\sum_{j=2}^{\pi_k} h_E^2 \|\diver \ww\|_{0,E}^2
\lesssim \|\ww\|_{0,E}^2 \,.
\end{aligned}
\end{gather}
Finally for the last term in \eqref{eq:st1}, using again
$$
\|\mmp\|_{L^{\infty}(E)}\leq 1,\qquad\text{and}\qquad\|m_l\|_{L^{\infty}(E)} \leq 1,
$$
we get:
\begin{gather}
\label{eq:st4}
\begin{aligned}
\sum_{l=1}^{\pi_{k -1}} |E|\boldsymbol{D_3}^{l}(\ww)^2 &=
\sum_{l=1}^{\pi_{k -1}} |E| \left(\frac{1}{|E|} \int_E \ww \cdot \mmp m_l \,
{\rm d}E\right)^2 \\
&\lesssim
\frac{1}{|E|}  \|\ww\|_{0,E}^2 \|\mmp m_l\|_{0,E}^2
\lesssim \|\ww\|_{0,E}^2 \,.
\end{aligned}
\end{gather}
Collecting \eqref{eq:st2}, \eqref{eq:st3} and \eqref{eq:st4} in \eqref{eq:st1} we obtain the thesis.
\end{proof}

The  next step is to prove the coercivity of the bilinear form $\St$ with respect to  the $L^2$-norm.
We start by noting that any function $\ww \in \VV_k(E)$ can be decomposed as
\begin{equation}
\label{eq:w1w2}
\ww = \nabla \phi - \ROT \psi
\end{equation}
where $\phi$ and $\psi$ are defined by
\begin{equation}
\label{eq:w1w2a}
\left \{
\begin{aligned}
&\Delta \phi = \diver \ww \, & \text{in $E$,}  \\
& \nabla \phi\cdot \nn = \ww \cdot \nn^E \, & \text{on $\partial E$,}
\end{aligned}
\right.
\qquad \,\,
\text{and}
\qquad \,\,
\left \{
\begin{aligned}
&\Delta \psi = \rot \ww \, & \text{in $E$,} \\
& \psi = 0 \, & \text{on $\partial E$,}
\end{aligned}
\right.
\end{equation}
we can assume that $\phi$ is zero averaged.
Moreover the decomposition is $L^2$-orthogonal, i.e.
\begin{equation}
\label{eq:pitagora}
\|\ww\|^2_{0, E} = \|\nabla \phi\|^2_{0, E} + \|\ROT \psi\|^2_{0, E} \, .
\end{equation}
%
\noindent
Given a vector $\gv \defeq (g_i)_{i=1}^N$, let $\|\gv\|^2_{l^2} \defeq \sum_{i=1}^N g_i^2$ be its Euclidean norm.
The following  lemma for polynomials is easy to check.
\begin{lemma}
\label{lm:l2piccolo}
Let $E \in \Omega_h$ and let $n \in \mathbb{N}$ a fixed integer.
Under Assumptions \ref{ass:regu}, \ref{ass:star} and~\ref{ass:mesh},
let $\gv \defeq (g_r)_{r=1}^{\pi_n}$ be a vector of real numbers and $g \defeq \sum_r^{\pi_n} g_r \, m_r \in \Pk_n(E)$, where $m_r \in \Mk_n(E)$.
Then we have the following norm equivalence
\[
h_E^2 \, \|\gv\|^2_{l^2} \lesssim \|g\|^2_{0, E} \lesssim  h_E^2 \, \|\gv\|^2_{l^2} \,.
\]
Moreover let $\gv \defeq (g_s)_{s=1}^{n+1}$ be a vector of real numbers and $\widetilde{g} \defeq \sum_s^{n+1} g_s \, \widetilde{m}_s \in \Pkt_n(e)$, where $\widetilde{m}_s \in \Mkt_n(e)$.
Then we have the following norm equivalence
\[
h_E \, \|\gv\|^2_{l^2} \lesssim \|g\|^2_{0, e} \lesssim  h_E \, \|\gv\|^2_{l^2} \,.
\]
\end{lemma}

\begin{prop}
\label{pr:coercivity}
Let $E \in \Omega_h$.
Under Assumptions \ref{ass:regu}, \ref{ass:star} and~\ref{ass:mesh} the following holds
\[
\|\ww\|_{0,E}^2 \lesssim  \St(\ww, \ww) \quad \forall \ww \in \VV_k(E).
\]
\end{prop}

\begin{proof}
Let $\ww \in \VV_k(E)$, since the decomposition \eqref{eq:w1w2} is $L^2$-orthogonal we need to prove that
\begin{equation}
\label{eq:co0}
\|\nabla \phi\|_{0,E}^2 \lesssim  \St(\ww, \ww)
\qquad \text{and} \qquad
\|\ROT \psi\|_{0,E}^2 \lesssim  \St(\ww, \ww) \,.
\end{equation}
We start with the first bound in \eqref{eq:co0} and we infer
\begin{gather}
\label{eq:co1}
\begin{gathered}
\|\nabla \phi\|_{0,E}^2 = \int_E  \ww \cdot \nabla \phi \, {\rm d}E   =
- \int_E \diver \ww \, \phi \, {\rm d}E + \sum_{E \in \EdE} \sigma_{E,e} \int_e
\ww \cdot \nn^e \phi \, {\rm d}e
\\
= - \int_E \diver \ww \, \Pi_0^k \phi \, {\rm d}E + \sum_{E \in \EdE}
\sigma_{E,e} \int_e \ww \cdot \nn^e \widetilde{\Pi}_0^k\phi \, {\rm d}e
\end{gathered}
\end{gather}
where in the last equation we use the fact that $\diver \ww \in \Pk_k(E)$ and $\ww \cdot
\nn^e \in \Pkt_k(e)$ and definitions \eqref{eq:projection_def} and
\eqref{eq:projection_edge}, respectively.
Let us set
\[
\Pi^k_0 \phi = \sum_{j=1}^{\pi_k} c_j m_j \text{ with } m_j \in \Mk_k(E),
\quad
\widetilde{\Pi}_0^k \phi = \sum_{i=1}^{k+1} d_i \widetilde{m}_i
\text{ with } \widetilde{m}_i \in \Mkt_k(e).
\]
Then from \eqref{eq:co1} we infer
\begin{multline*}
\|\nabla \phi\|_{0,E}^2 = -\sum_{j=2}^{\pi_k} c_j \int_E \diver \ww \, m_j \, {\rm d}E
- c_1 \int_{\partial E} \ww \cdot \nn^E m_j \, {\rm d}e  + \\+
\sum_{E \in \EdE} \sigma_{E,e} \sum_{i=1}^{k+1} d_i \int_e \ww \cdot \nn^e
\widetilde{m}_i \, {\rm d}e
\end{multline*}
that is
\begin{gather}
\label{eq:co2}
\|\nabla \phi\|_{0,E}^2 = -\sum_{j=2}^{\pi_k} c_j \int_E \diver \ww \, m_j \, {\rm d}E +
\sum_{E \in \EdE} \sigma_{E,e} \sum_{i=1}^{k+1} \hat{d}_i \int_e \ww \cdot \nn^e
\widetilde{m}_i \, {\rm d}e
\end{gather}
where $\hat{d}_i = d_i - c_i$ if $i=1$, $\hat{d}_i= d_i$ otherwise.

Using Lemma \ref{lm:l2piccolo}, the continuity of $\Pi_0^k$ with respect to the $L^2$-norm and a scaled Poincar\'e inequality for the zero averaged function $\phi$, the bulk integral in \eqref{eq:co2} can be bounded as follows:
\begin{gather}
\label{eq:co3}
\begin{aligned}
-\sum_{j=2}^{\pi_k} c_j  & \int_E \diver \ww \, m_j \, {\rm d}E =
- \sum_{j=2}^{\pi_k} c_j \, \frac{|E|}{h_E} \, \boldsymbol{D_2}^j(\ww)
\\
& \lesssim  \biggl( \sum_{j=1}^{\pi_k} c_j^2  \biggr)^{1/2} \biggl( |E| \sum_{j=2}^{\pi_k}  \boldsymbol{D_2}^i(\ww)^2 \biggr)^{1/2}
 \lesssim  h_E^{-1} \, \|\Pi_0^k \phi\|_{0,E} \, \St(\ww, \ww)^{1/2}\\
&\lesssim h_E^{-1} \, \|\phi\|_{0,E} \, \St(\ww, \ww)^{1/2}
\lesssim \|\nabla \phi\|_{0,E} \, \St(\ww, \ww)^{1/2} \,.
\end{aligned}
\end{gather}
For the boundary integral in \eqref{eq:co2}, employing Lemma \ref{lm:l2piccolo}, we infer
\[
\begin{aligned}
\sum_{e \in \EdE}  \sigma_{E,e} &\sum_{i=1}^{k+1} \hat{d}_i \int_e \ww \cdot
\nn^e \widetilde{m}_i \, {\rm d}e
 \lesssim
\sum_{e \in \EdE} \sigma_{E,e} \sum_{i=1}^{k+1} \hat{d}_i h_e \boldsymbol{D_1}^{e,i}(\ww)
\\
&\lesssim
\sum_{e \in \EdE} \biggl(\sum_{i=1}^{k+1} \hat{d}_i^2 \biggr)^{1/2}
\biggl( |E| \sum_{i=1}^{k+1} \boldsymbol{D_1}^{e,i}(\ww) \biggr)^{1/2}
\\
&\lesssim
\sum_{e \in \EdE} \biggl(c_1^2 + \sum_{i=1}^{k+1} d_i^2\biggr)^{1/2}
\biggl( |E| \sum_{i=1}^{k+1} \boldsymbol{D_1}^{e,i}(\ww) \biggr)^{1/2}\\
&
\lesssim
\sum_{e \in \EdE} \biggl(h_E^{-1}\| \Pi_0^k \phi\|_{0,E} +  h_E^{-1/2}
\|\widetilde{\Pi}_0^k \phi\|_{0,e} \biggr)
\biggl( |E| \sum_{i=1}^{k+1} \boldsymbol{D_1}^{e,i}(\ww)^2 \biggr)^{1/2} \,.
\end{aligned}
\]
Then, using the continuity of $\widetilde{\Pi}_0^k$ with respect to the $L^2$-norm and the $H^1$ trace inequality for the zero averaged function $\phi$, from  previous bound we get
\begin{gather}
\label{eq:co4}
\begin{aligned}
&\sum_{e \in \EdE}  \sigma_{E,e} \sum_{i=1}^{k+1} \hat{d}_i \int_e \ww \cdot
\nn^e \widetilde{m}_i \, {\rm d}e
\\
&\lesssim
\biggl( h_E^{-2}\| \Pi_0^k \phi\|^2_{0,E} +  h_E^{-1}  \sum_{e \in \EdE}
\|\widetilde{\Pi}_0^k \phi\|_{0,e}^2 \biggr)^{1/2}
\biggl( |E| \sum_{e \in \EdE}\sum_{i=1}^{k+1} \boldsymbol{D_1}^{e,i}(\ww)^2 \biggr)^{1/2}
\\
& \lesssim
\biggl( h_E^{-2}\| \phi\|^2_{0,E} +  h_E^{-1}  \sum_{e \in \EdE} \|\phi\|_{0,e}^2 \biggr)^{1/2} \,  \St(\ww, \ww)^{1/2}
 \\
& \lesssim
\biggl( h_E^{-2}\| \phi\|^2_{0,E} +  h_E^{-1}  \|\phi\|_{0,\partial E}^2 \biggr)^{1/2} \,  \St(\ww, \ww)^{1/2}
\lesssim \|\nabla \phi\|_{0,E} \, \St(\ww, \ww)^{1/2} \,.
\end{aligned}
\end{gather}
Collecting \eqref{eq:co4} and \eqref{eq:co3} in \eqref{eq:co2},  we obtain the first bound in \eqref{eq:co0}.

Concerning the $\ROT$ part of $\ww$ in decomposition \eqref{eq:w1w2}, recalling \eqref{eq:w1w2a}, we infer
\begin{equation}
\label{eq:co7}
\begin{split}
\|\ROT \psi\|_{0,E}^2 &= \int_E \ROT \psi \cdot \ROT \psi \, {\rm d}E = \int_E \Delta \psi \, \psi \, {\rm d}E =  \int_E \rot \ww \, \psi \, {\rm d}E \,.
\end{split}
\end{equation}
Since $\rot \ww = q_{k-1} \in \Pk_{k-1}(E)$ there exists $p_{k-1} \in \Pk_{k-1}(E)$ such that
$\rot \ww = \rot(\xxp p_{k-1})$ (cfr. Remark \ref{rm:rot}).
Moreover being $\ROT \psi$ orthogonal with respect to the gradients, by decomposition \eqref{eq:polydec}, it holds $\xxp p_{k-1}= \Pi_0^k \ROT \psi$.
Therefore from \eqref{eq:co7} we obtain
\begin{gather}
\label{eq:co8}
\begin{aligned}
\|\ROT \psi\|_{0,E}^2 &= \int_E \rot(\xxp p_{k-1})\, \psi \, {\rm d}E =
\int_E \xxp p_{k-1}\cdot \ROT \psi \, {\rm d}E
\\
&= \int_E \xxp p_{k-1} \cdot \ww  \, {\rm d}E -
\int_E \xxp p_{k-1} \cdot \nabla \phi  \, {\rm d}E \,.
\end{aligned}
\end{gather}
Let us write $\xxp p_{k-1}$ in the monomial basis: it exists $g_l \in
\mathbb{R}$, for $l=1, \ldots, \pi_{k-1}$, such that
\[
\xxp p_{k-1} \defeq \sum_{l=1}^{\pi_{k-1}} g_l \mmp m_l \,,
\]
and let us analyse the two adds in the right-hand side of \eqref{eq:co8}.
For the first one, using Lemma \ref{lm:l2piccolo},  we infer
\begin{gather}
\label{eq:co9}
\begin{aligned}
\int_E \xxp p_{k-1} \,\cdot \ww   \, {\rm d}E &=
\sum_{l=1}^{\pi_{k-1}} g_l \int_E \mmp m_l \,\cdot \ww  \, {\rm d}E =
\sum_{l=1}^{\pi_{k-1}} |E| g_l \boldsymbol{D_3}^l(\ww)
\\
& \lesssim h_E \biggl(\sum_{l=1}^{\pi_{k-1}} g_l^2 \biggr)^{1/2} \biggl( |E|\sum_{l=1}^{\pi_{k-1}} \boldsymbol{D_3}^l(\ww)^2\biggr)^{1/2}
\\
& \lesssim
\|\xxp p_{k-1}\|_{0,E} \, \St(\ww, \ww)^{1/2}
\lesssim
\|\ROT \psi\|_{0,E} \, \St(\ww, \ww)^{1/2}\,.
\end{aligned}
\end{gather}
For the second term in \eqref{eq:co8}, using the first bound in \eqref{eq:co0}, we get
\begin{gather}
\label{eq:co10}
\begin{aligned}
\int_E \xxp p_{k-1} \cdot \nabla \phi   \, {\rm d}E &\lesssim \|\xxp p_{k-1}\|_{0,E} \|\nabla \phi\|_{0,E} \lesssim
\|\xxp p_{k-1}\|_{0,E} \, \St(\ww, \ww)^{1/2}
\\
&\lesssim
\|\ROT \psi\|_{0,E} \, \St(\ww, \ww)^{1/2}\,.
\end{aligned}
\end{gather}
Collecting \eqref{eq:co9} and \eqref{eq:co8} in \eqref{eq:co10} we obtain the second bound in \eqref{eq:co0}.
The thesis now follows from \eqref{eq:pitagora}.
\end{proof}
As a direct consequence of Proposition \ref{pr:interpolation}, Proposition \ref{pr:infsup}, Proposition \ref{pr:continuity} and Proposition \ref{pr:coercivity} we have the following result \cite{Brezzi2014,BeiraoVeiga2016}.
\begin{prop}
\label{pr:final}
Under Assumptions \ref{ass:regu}, \ref{ass:star} and \ref{ass:mesh},
the virtual element problem \eqref{pb:darcy_vem} has a unique solution $(\bm{q}_h, p_h) \in \VV_k(\Omega_h) \times Q_k(\Omega_h)$.
Moreover, let $(\bm{q}, p) \in \VV \times Q$ be the solution of problem \eqref{pb:darcy_weak} and assume that $\bm{q} \in [H^{k+1}(\Omega_h)]^2$ with $\diver \bm{q} \in H^{k+1}(\Omega_h)$,  $p$, $f \in H^{k+1}(\Omega_h)$, then the following error estimates hold:
\[
\begin{gathered}
\|\bm{q} - \bm{q}_h\|_{\VV} \lesssim h^{k+1}(|\bm{q}|_{k+1, \Omega_h} + |f|_{k+1, \Omega_h})\,,
\\
\|p - p_h\|_{Q} \lesssim h^{k+1} (|\bm{q}|_{k+1, \Omega_h} + |p|_{k+1,\Omega_h})\,.
\end{gathered}
\]
\end{prop}

}
\newcommand{\dE}{\text{d}E}
\newcommand{\Or}[1]{O\left(h^{#1}\right)}

\section{Numerical tests}
\label{sec:numExe}
In this section some numerical examples are provided to describe the behaviour of the method and give numerical evidence of the theoretical results derived in the previous sections.
More specifically, we propose a comparison of the method with standard mixed virtual elements, in which the curved boundaries or interfaces of the domains are approximated by a straight edge interpolant. For brevity we will label the present approach which honours domain geometry as \texttt{withGeo}, and the standard approach as \texttt{noGeo}.


We use the projection operators introduced in~\eqref{eq:projection_def} to define
the following error indicators for both variables; for a given exact solution $(\bm{q}, p)$ of Problem~\ref{pb:darcy_model},
we compute:
\begin{itemize}
 \item \textbf{velocity $L^2$ error:}
 $$
    e_{\bm{q}}^2 \defeq {\sum_{E\in\Omega_h} \|\bm{q} - {\Pi}^k_0\bm{q}_h\|^2_E}\,,
 $$
 \item \textbf{pressure $L^2$ error:}
 $$
    e_p^2 \defeq {\sum_{E\in\Omega_h} \|p - p_h\|^2_E}\,.
 $$
\end{itemize}
Moreover, to proceed with the convergence analysis, we define the mesh-size parameter
\begin{equation*}
h = \frac{1}{L_E}\sum_{E\in\Omega_h} h_E\,,
\end{equation*}
For each test we build a sequence of four meshes with decreasing mesh size parameter $h$ and the trend of each error indicator is computed and compared to the
expected convergence trend, which, for sufficiently regular data is $\Or{k+1}$ in accordance to Proposition \ref{pr:final}.

\subsection{Curved boundary}\label{sub:bound}

\paragraph{Problem description}
In this subsection we consider Problem~\ref{pb:darcy_model} on the domain $\Omega$
shown in Figure~\ref{fig:domExe1}.
Such domain is obtained from the unit square $(0,\,1)^2$ deforming the top and the bottom edges to make them curvilinear, i.e.,
they are the graph of the following cubic functions:
$$
g_1(x) = \frac{1}{2}x^2(x-1)+1\qquad\text{and}\qquad g_2(x) = \frac{1}{2}x^2(x-1)\,.
$$
We set the right hand side and the boundary conditions
in such a way that the exact solution of Problem~\ref{pb:darcy_model} is the couple:
\begin{equation*}
\bm{q}(x,\,y) = \left(\begin{array}{r}
                 \pi\,\cos(\pi\,x)\,\cos(\pi\,y)\\
                 -\pi\,\sin(\pi\,x)\,\sin(\pi\,y)
                 \end{array}\right)
\qquad\text{and}\qquad
p(x,\,y) = \sin(\pi x)\,\cos(\pi y)\,.
\end{equation*}

In this first example we take $\mu=1.$ and 
we consider a constant tensor $\K= \mathbb{I}$, 
where $\mathbb{I}$ is the identity matrix.

\begin{figure}[!htb]
\centering
\includegraphics[width=0.20\textwidth]{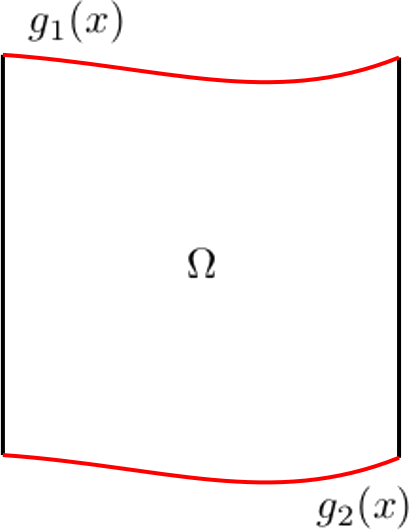}
\caption{Curved boundary: domain $\Omega$ considered in such example,
curved boundaries are highlighted in red.}
\label{fig:domExe1}
\end{figure}

\paragraph{Meshes}
Computational meshes are obtained starting from polygonal meshes defined on the unit square $(0,1)^2$ and subsequently modified, following the idea proposed in~\cite{BeiraodaVeiga2019}.
In the present case, {only} the $y$-component of a generic point $P$ is modified, i.e.,
the point $P(x_P,\,y_P)$ becomes $P'(x_P',\,y_P')$ where
$$
{
x_P' = x_P\qquad\text{and}\qquad
y_P' = \begin{cases}
    y_P + g_2(x_P) (1 - 2y_P) &\text{if}\:y_P\leq 0.5\\[0.5em]
    1 - y_P + g_1(x_P)\,(2y_P-1) &\text{if}\:y_P> 0.5
    \end{cases}\,.
}
$$
{The curved part of the boundary is further exactly reproduced for the \texttt{withGeo} case.}
As initial meshes we consider the following types of discretization of the unit square:
\textit{i)} \texttt{quad}, a uniform mesh composed by squares;
\textit{ii)} \texttt{hexR}, a mesh composed by hexagons;
\textit{iii)} \texttt{hexD}, a mesh composed by distorted hexagons;
\textit{iv)} \texttt{voro}, a centroidal Voronoi tessellation.
The last two types of meshes have some interesting features which challenge the robustness of the virtual element approach:
in particular $\texttt{hexD}$ meshes have distorted elements, whereas
$\texttt{voro}$ meshes have tiny edges, see Figure~\ref{fig:meshGene}.

\begin{figure}[!htb]
\centering
\begin{tabular}{cccc}
\texttt{quad} &\texttt{hexR} & \texttt{hexD} &\texttt{voro} \\
\includegraphics[width=0.21\textwidth]{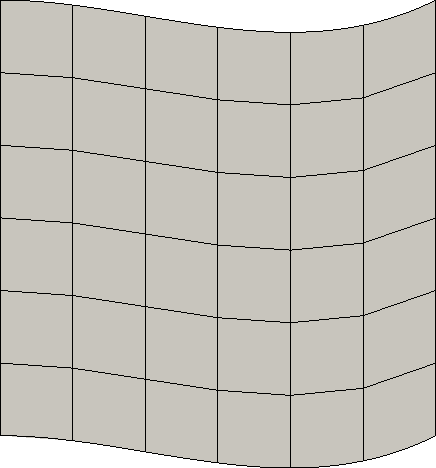} &
\includegraphics[width=0.21\textwidth]{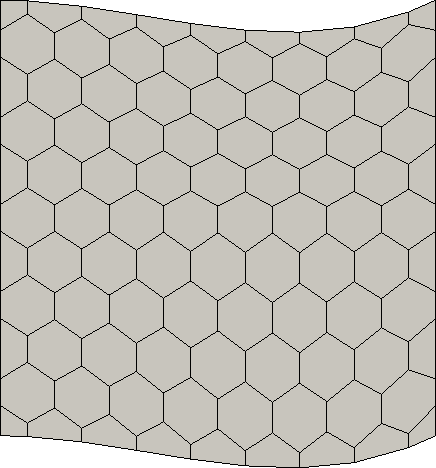}  &
\includegraphics[width=0.21\textwidth]{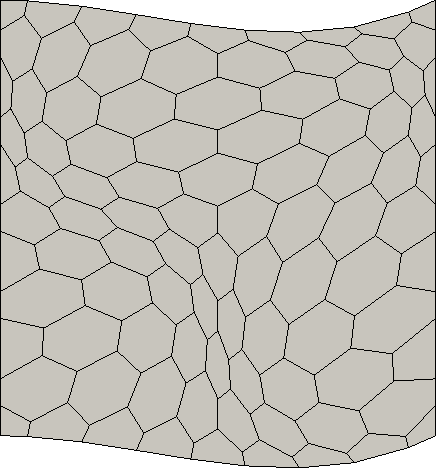} &
\includegraphics[width=0.21\textwidth]{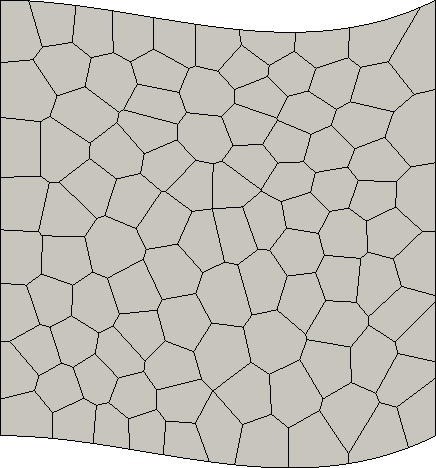} \\
\end{tabular}
\caption{Curved boundary: types of discretization used to proceed with the convergence analysis.}
\label{fig:meshGene}
\end{figure}

\paragraph{Results}
In Figures~\ref{fig:convExe1Quad},~\ref{fig:convExe1Bee},~\ref{fig:convExe1Hexa} and~\ref{fig:convExe1Voro},
we collect the results for the various types of meshes.
The reported convergence lines of the \texttt{withGeo} and \texttt{noGeo} approaches coincide for polynomial degrees $k=0$ and 1. 
They have the expected convergence rate of $\Or{1}$ and $\Or{2}$, respectively.
On the contrary, 
for polynomial degree $k>1$ the trend of both velocity and pressure $L^2$ errors is different between the two strategies.

More specifically, the convergence trends of the \texttt{noGeo} case is bounded by the geometrical representation error to $\Or{2}$, as this error dominates the accuracy of the approximation with mixed virtual elements. On the contrary the proposed approximation scheme \texttt{withGeo} behaves as expected for both velocity and pressure variables and for each approximation degree, showing the optimal convergence trend for the used polynomial degree.
Such behaviour is in line to what observed in~\cite{BeiraodaVeiga2019} for a Laplace problem.

\begin{figure}[!htb]
\centering
\begin{tabular}{cc}
\multicolumn{2}{c}{\texttt{quad}} \\
\includegraphics[width=0.49\textwidth]{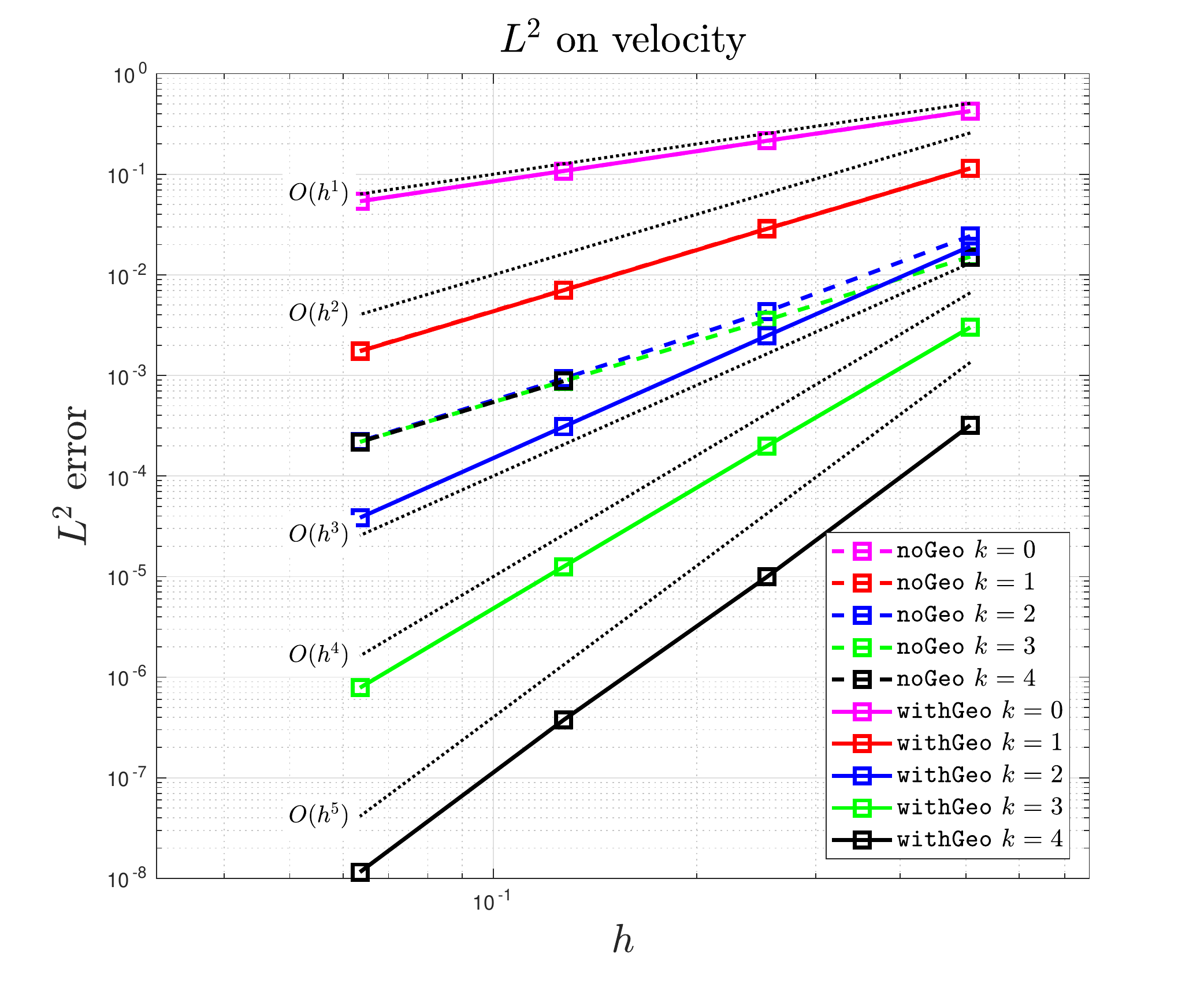} &
\includegraphics[width=0.49\textwidth]{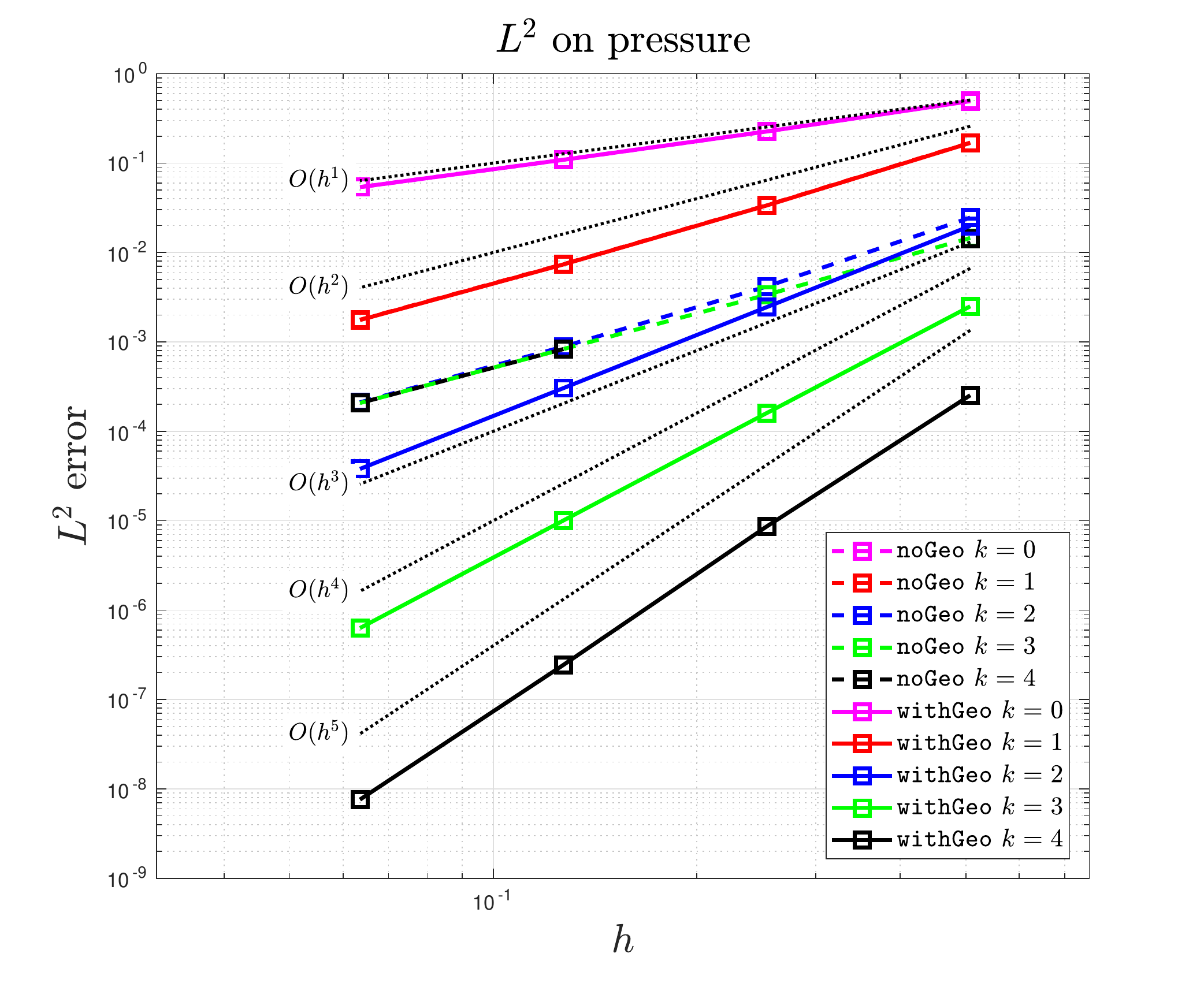}\\
\end{tabular}
\caption{Curved boundary: convergence lines for \texttt{quad} meshes for each VEM approximation degrees.}
\label{fig:convExe1Quad}
\end{figure}

\begin{figure}[!htb]
\centering
\begin{tabular}{cc}
\multicolumn{2}{c}{\texttt{hexR}} \\
\includegraphics[width=0.49\textwidth]{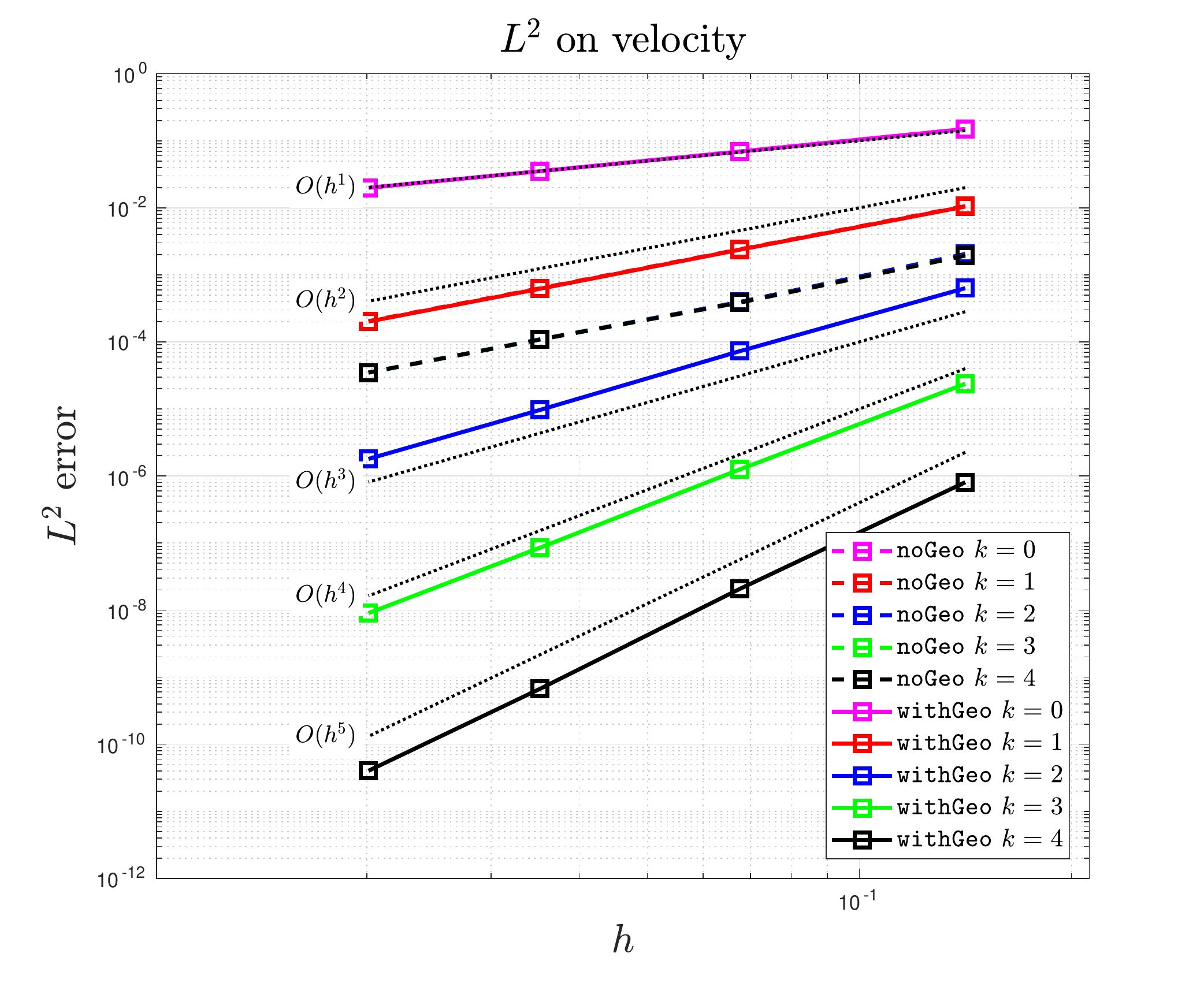} &
\includegraphics[width=0.49\textwidth]{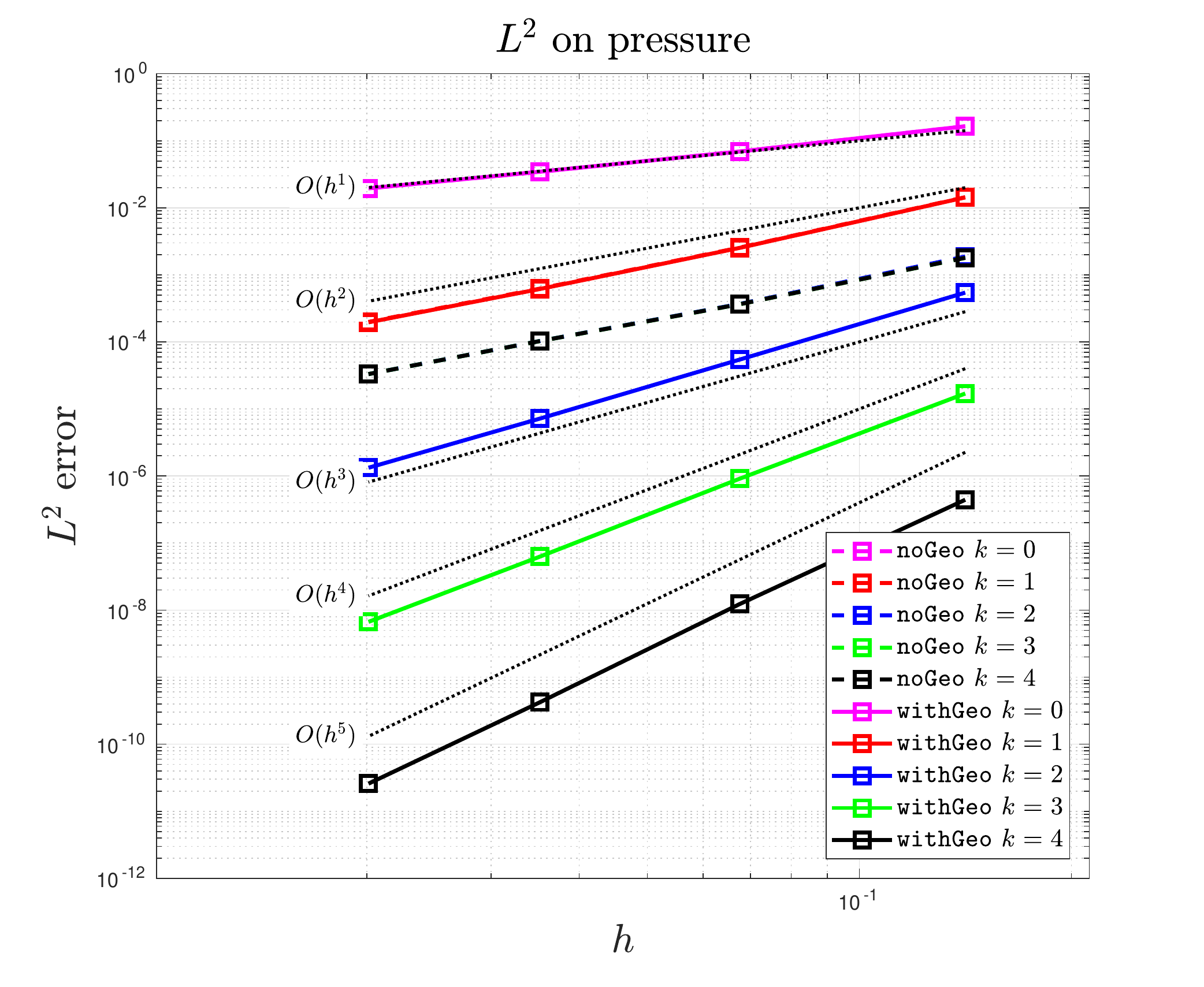}\\
\end{tabular}
\caption{Curved boundary: convergence lines for \texttt{hexR} meshes for each VEM approximation degrees.}
\label{fig:convExe1Bee}
\end{figure}

\begin{figure}[!htb]
\centering
\begin{tabular}{cc}
\multicolumn{2}{c}{\texttt{hexD}} \\
\includegraphics[width=0.49\textwidth]{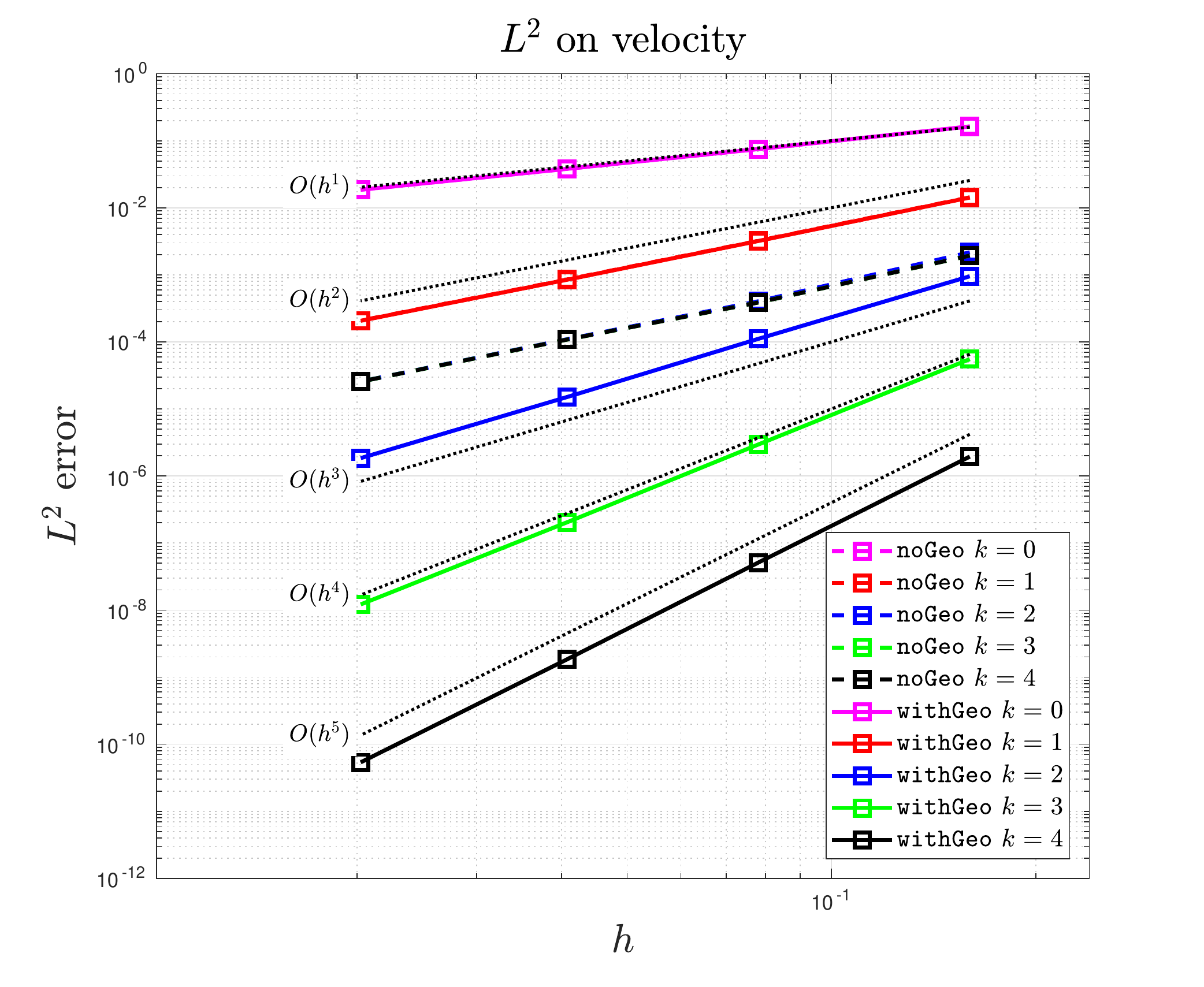} &
\includegraphics[width=0.49\textwidth]{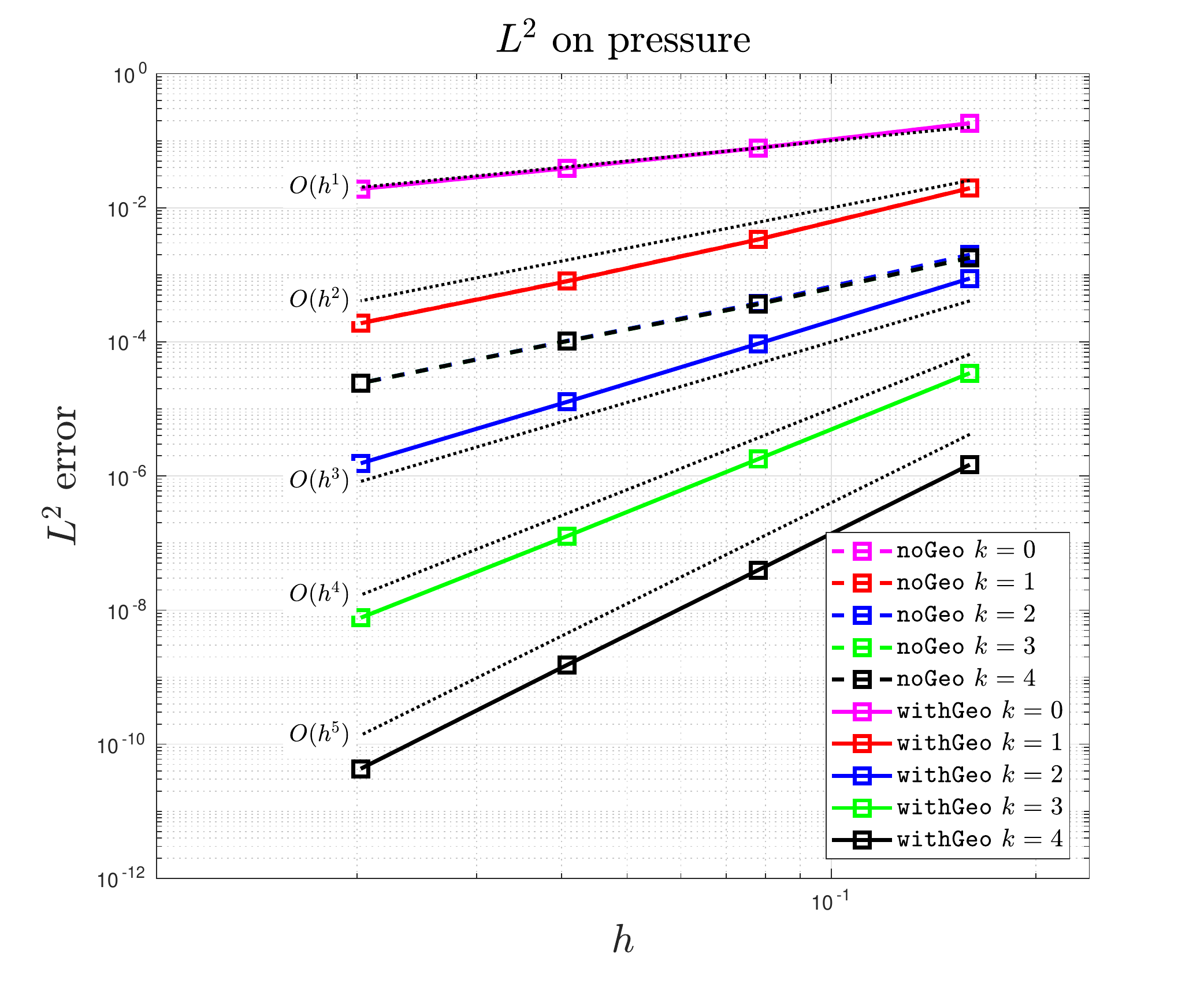}\\
\end{tabular}
\caption{Curved boundary: convergence lines for \texttt{hexD} meshes for each VEM approximation degrees.}
\label{fig:convExe1Hexa}
\end{figure}

\begin{figure}[!htb]
\centering
\begin{tabular}{cc}
\multicolumn{2}{c}{\texttt{voro}} \\
\includegraphics[width=0.49\textwidth]{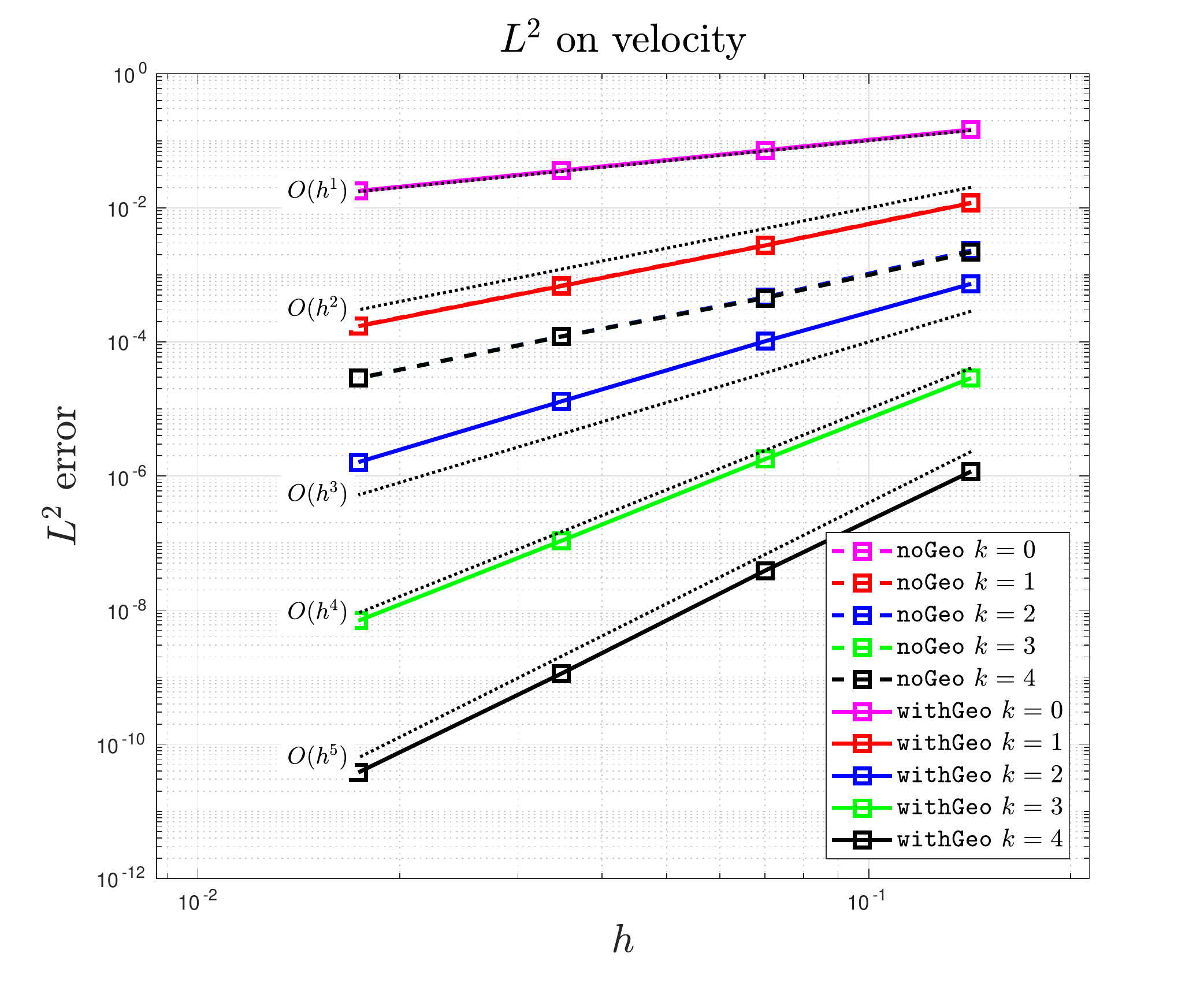} &
\includegraphics[width=0.49\textwidth]{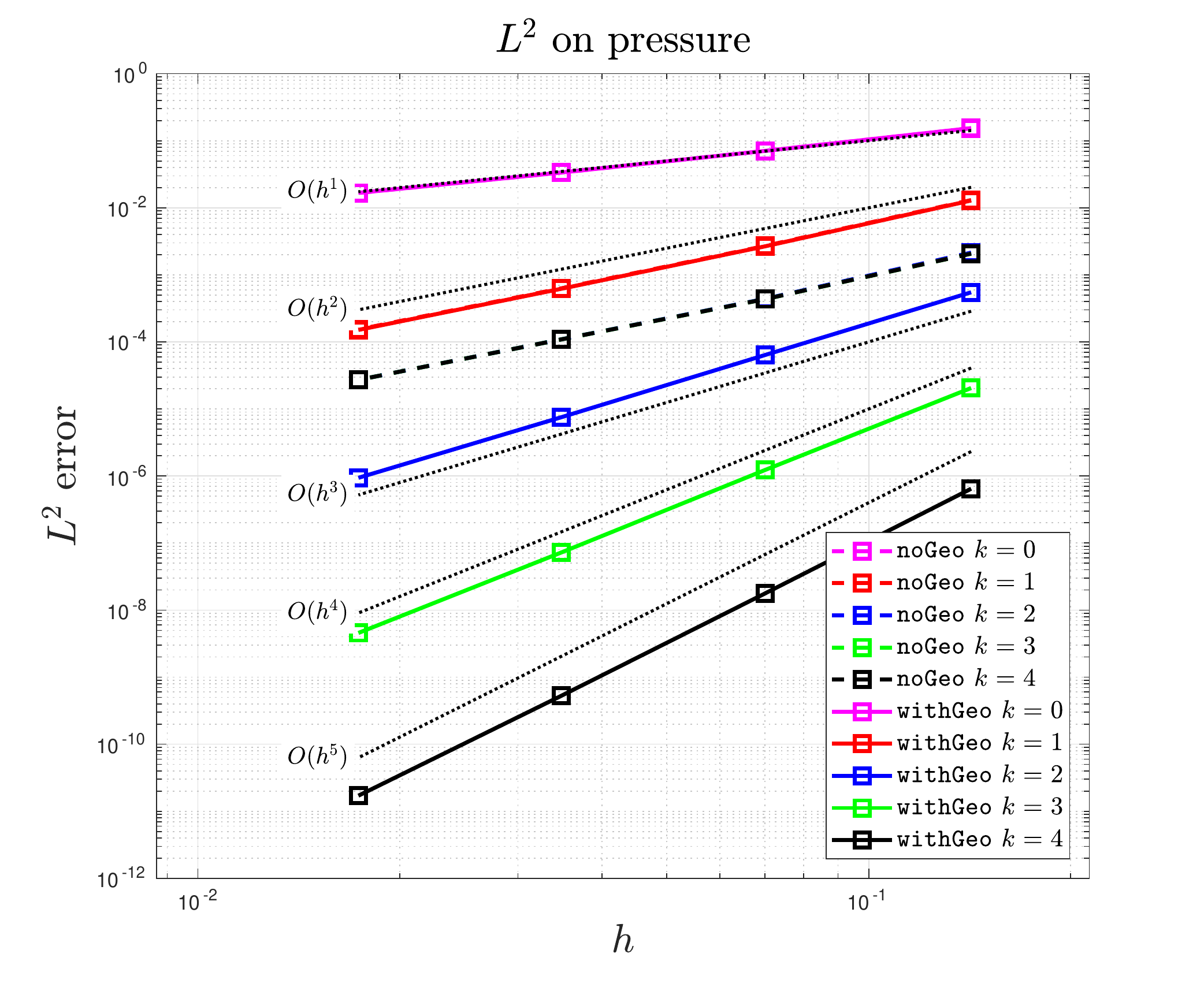}
\end{tabular}
\caption{Curved boundary: convergence lines for \texttt{voro} meshes for each VEM approximation degrees.}
\label{fig:convExe1Voro}
\end{figure}

%
%
%
%

\subsection{Internal curved interface}\label{sub:inside1}

\paragraph{Problem description}
In this subsection we consider again Problem~\ref{pb:darcy_model} defined on a different domain with respect to the previous example. The domain $\Omega$ is
shown in Figure~\ref{fig:domExe2} and consists of a unit square $\Omega =
(-1,\,1)^2$, $\Omega=\overline{\Omega_1}\cup\overline{\Omega_2}$, being $\Omega_2$ a circular inclusion with radius $R=0.45$ and $\Omega_1:=\Omega\backslash\Omega_2$ a circular crown.
Two different values of the tensor $\K= k \mathbb{I}$ are prescribed on each subdomain:
$k_1 = 1$ and $k_2 = 0.1$
for the subdomain $\Omega_1$ and $\Omega_2$, respectively,
while $\mu=1.$ on each subdomain.
We set the right hand side and the boundary conditions
in such a way that the exact solution for the pressure is
\begin{equation*}
p_1(x,\,y) = k_2\cos\left(\sqrt{x^2+y^2}\right) + \cos(R)\,(1 - k_2)
\end{equation*}
and
\begin{equation*}
p_2(x,\,y) = \cos\left(\sqrt{x^2+y^2}\right)\,,
\end{equation*}
for the subdomains $\Omega_1$ and $\Omega_2$, respectively.
Then, the exact solution for the velocity variable is given by
$$
\bm{q}_i(x,\,y) = -k_i(x,\,y) \nabla p_i(x,\,y)\,,\qquad\text{for }i=1,2\,.
$$
The pressure solution is chosen in such a way that we have a $C^0$ continuity on $\partial \Omega_2$,
and the velocity field has a $C^0$ continuity of the normal component across $\partial \Omega_2$, i.e.,
$$
p_1= p_2\qquad\text{and}\qquad \bm{q}_1\cdot\bm{n}_\iota + \bm{q}_2\cdot
\bm{n}_\iota = 0\qquad\text{on}\,\,\partial\Omega_2\,,
$$
where $\bm{n}_\iota$ is the normal of $\partial \Omega_2$ pointing from $\Omega_1$ to
$\Omega_2$.

\begin{figure}[!htb]
\centering
\subfloat[Domain]{\includegraphics[width=0.325\textwidth]{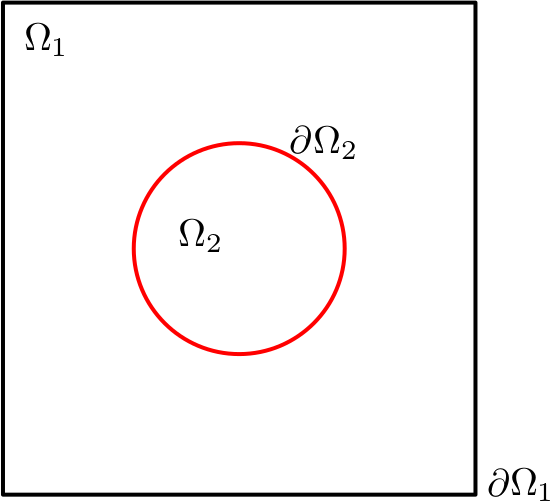}\label{fig:domExe2}}%
\hspace*{0.1\textwidth}%
\subfloat[Mesh]{\includegraphics[width=0.3\textwidth]{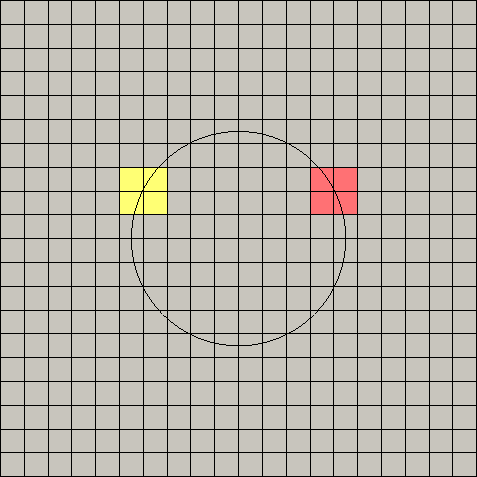}\label{fig:exe2Mesh}}%
\caption{Internal curved interface. On the left, domain $\Omega$ considered in such example,
curved boundaries are highlighted in red. On the right, the whole mesh with the internal curved boundary.
We show zooms of the yellow and red regions in Figure~\ref{fig:exe2MeshZoom}.}

\end{figure}

\paragraph{Meshes}
To generate the grid, we start again from a structured mesh composed of square
elements of the whole domain $\Omega$, independently of the internal interface
$\partial\Omega_2$, and \emph{then} we cut the mesh elements into sub-elements
according to $\partial\Omega_2$. {The geometry of the internal interface is
exactly reproduced in the proposed \texttt{withGeo} approach, whereas it is
replaced by straight edges in the \texttt{noGeo} approach}. In both cases,
thanks to the ability of virtual elements in dealing with arbitrary shaped
elements the mesh generation process is straightforward, as we do not need to
re-mesh elements crossed by the circle, but we simply cut each intersected
quadrilateral element into two new elements with one new (curved) edge,
\emph{without} taking care about the resulting shape and size of the two cut
elements.


The flexibility in including interfaces in the mesh and the robustness with respect to element size/distortion
are a huge advantage from the mesh generation point of view. In many applications a large number of possibly intersecting interfaces might be present in the computational domain, such that a robust and easy mesh generation process is of paramount importance. In such cases, the generation of good quality triangular meshes constrained to the interfaces might be an extremely complex task which might result in overly refined regions of the mesh, only needed to honour the geometry of the interfaces, independently from the desired accuracy level.
If we consider a virtual element approach, interfaces can be easily superimposed to an existing regular mesh, as shown above, avoid unnecessary refinements and consequently decreasing the degrees of freedom and the computational effort.

\begin{figure}[!htb]
\centering
\includegraphics[width=0.3035\textwidth]{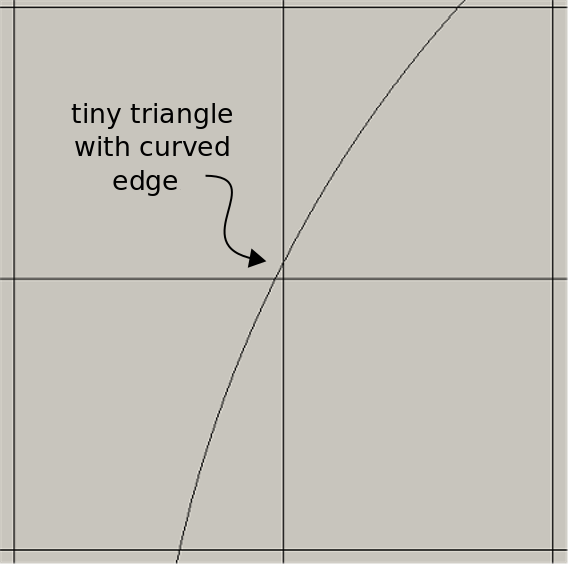}%
\hspace*{0.1\textwidth}%
\includegraphics[width=0.3035\textwidth]{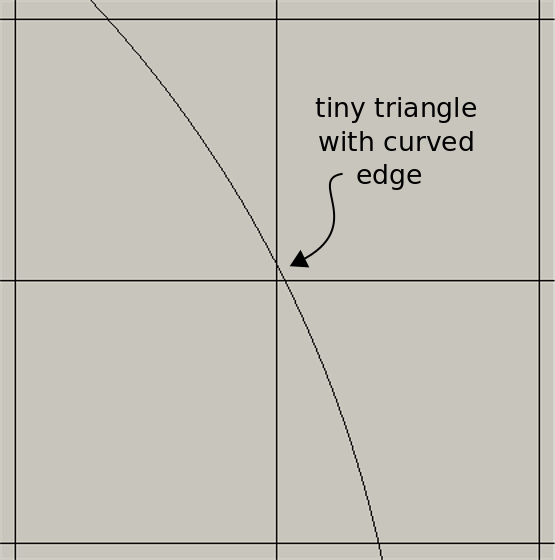}
\caption{Internal curved interface: zooms of the yellow and red regions of Figure~\ref{fig:exe2Mesh},
where we highlight tiny triangles with a curved edge.}
\label{fig:exe2MeshZoom}
\end{figure}

In Figure~\ref{fig:exe2MeshZoom} we show a detail of some cut elements.
Here we better appreciate that elements crossed by the interface are simply split in two parts and
there is no any further subdivision. Moreover, we notice that such meshing procedure might results in really tiny elements adjacent to big ones.
In each mesh of the following convergence analysis there are many elements with these characteristics and
we will see that the convergence trend of the method is not affected by them.

As a final remark, we would like to underline another interesting property of the proposed approach.
The proposed curved spaces are compatible with standard finite element discretizations. 
For instance it is possible to simply glue a standard Raviart-Thomas element with an element with curved edges along a straight edges, thus exploiting the proposed virtual element spaces \emph{only} on the elements with curvilinear edges and standard Raviart-Thomas discretization on elements with straight edges.

As we have done for the previous example we make a sequence of four meshes with decreasing mesh size $h$
to proceed with the convergence analysis.

\paragraph{Results}
In Figure~\ref{fig:exe2Res} we show the convergence lines for the
\texttt{withGeo} and \texttt{noGeo} approaches as $h$ is reduced, for values of
$k$ ranging between $0$ and $4$.  The behaviour of the error is similar to the
one shown in the previous example.  Indeed, in the \texttt{noGeo} case the
convergence is the optimal one for polynomial accuracy values $k=0$ and $1$,
while for $k>1$ the geometrical error dominates the VEM approximation error and
the trend remains bounded by $\Or{2}$.  On the contrary, when we consider the
virtual element spaces for curvilinear edges, optimal error decay $\Or{k+1}$ is
obtained for both velocity and pressure $L^2$ errors, for the used polynomial
accuracy $k$. {A pre-asymptotic behaviour is observed for the \texttt{withGeo}
approach for values of $k=2,3$ and 4, which however terminates in the considered
range of $h$ values for almost all cases.}


\begin{figure}[!htb]
\centering
\includegraphics[width=0.49\textwidth]{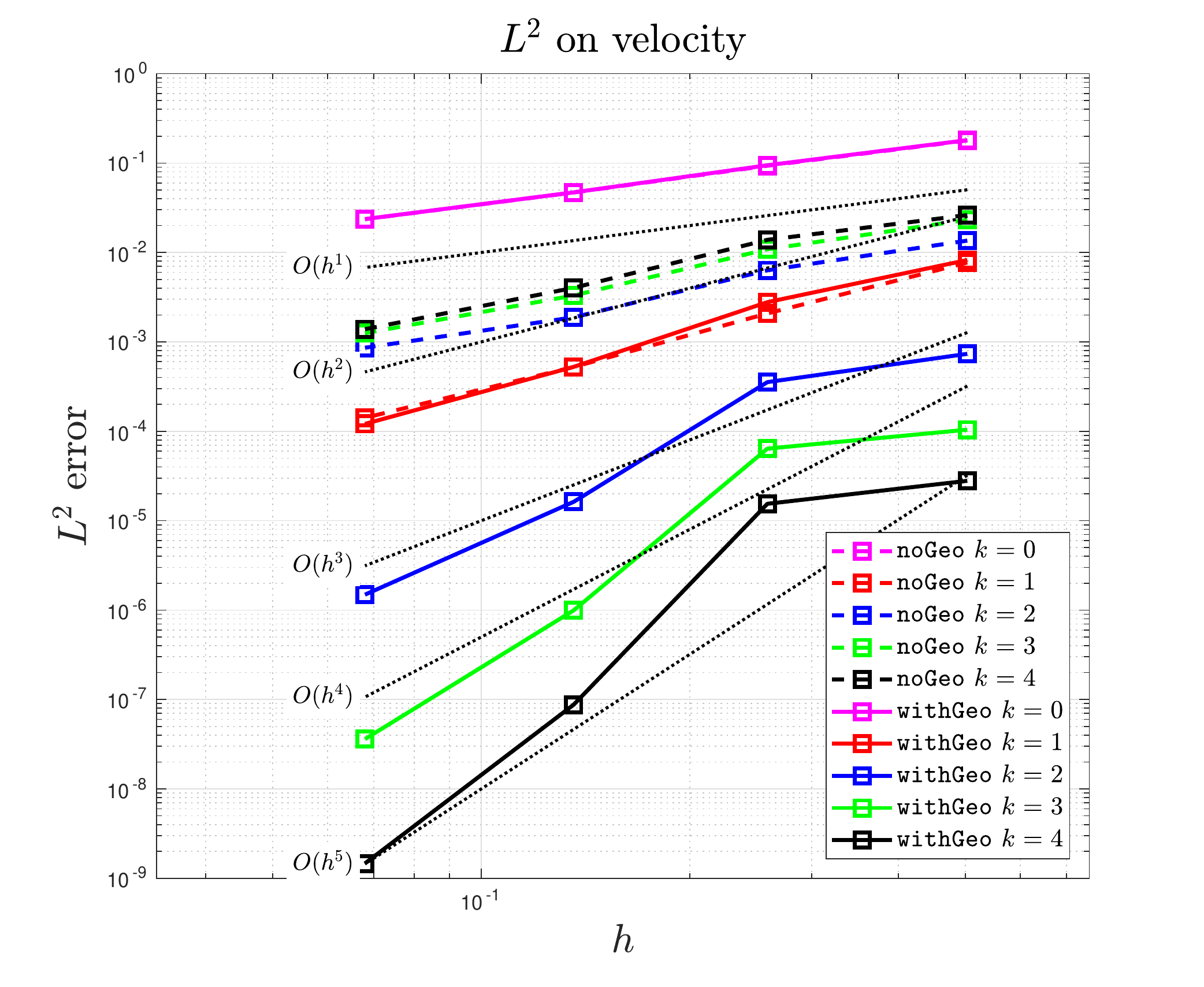}%
\includegraphics[width=0.49\textwidth]{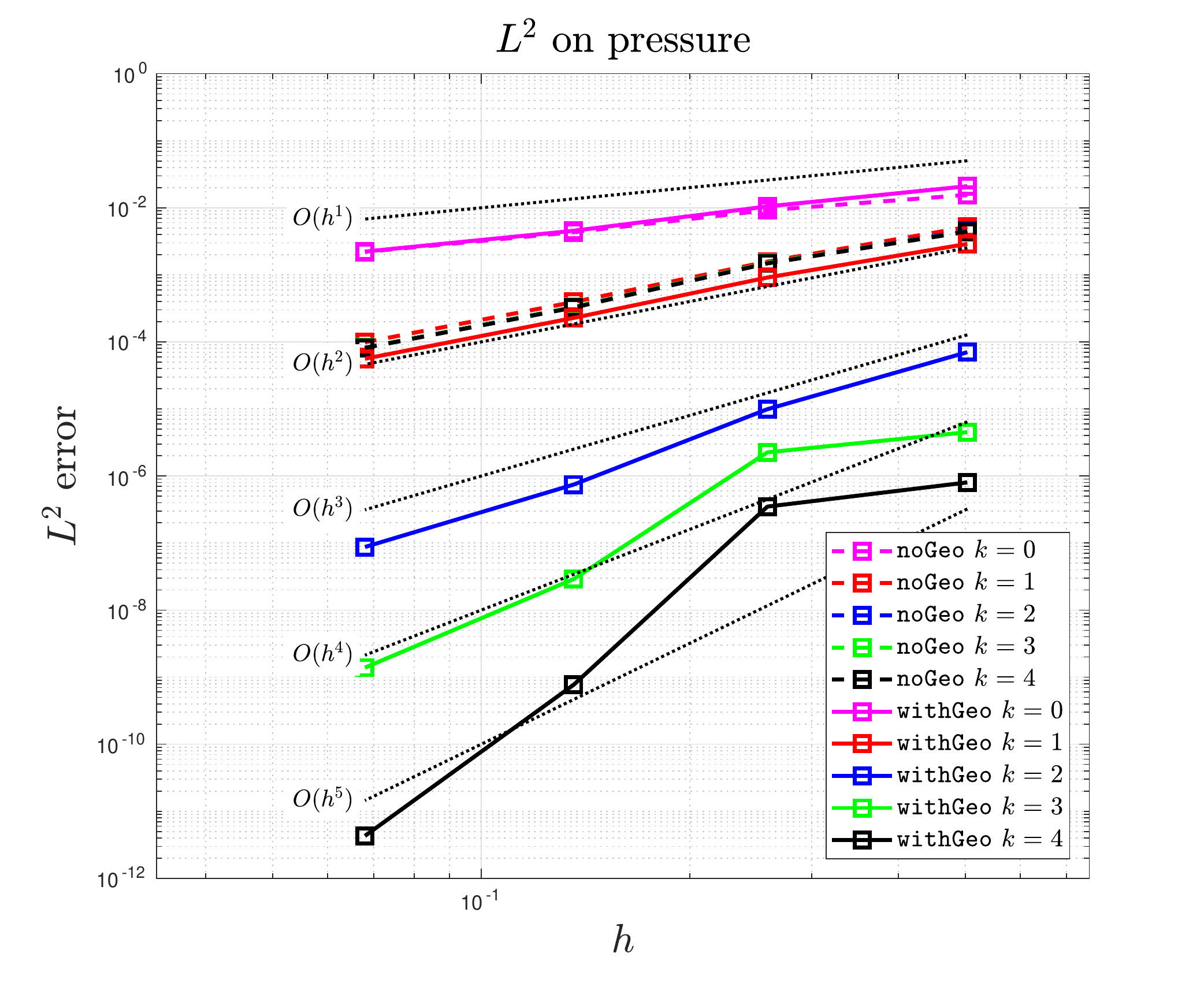}
\caption{Internal curved interface: convergence lines for each VEM approximation degrees.}
\label{fig:exe2Res}
\end{figure}

\subsection{Double internal curved interfaces}\label{sub:inside2}

\paragraph{Problem description}
In this example we consider two internal boundaries which identify three regions, $\Omega_1$, $\Omega_2$ and $\Omega_3$,
inside the square  $(-1,\,1)^2$, see Figure~\ref{fig:domExe3}.
Both internal boundaries are curved, i.e.,
$\Gamma_1$ and $\Gamma_2$ are defined as
$$
g_1(x) = a\sin(\pi x) + b\qquad\text{and}\qquad g_2(x) = a\sin(\pi x) - b\,,
$$
respectively. For this example we set $a=0.2$ and $b=0.31$.
Then, we set the right hand side of Problem~\ref{pb:darcy_model}
in such a way that the pressure solution is
\begin{eqnarray*}
p_1(x,\,y) &=& a\sin(\pi x)\,,\\
p_2(x,\,y) &=& a\,\sin\left\{\frac{\pi}{2b}[y-a \sin(\pi x)]\right\}\sin(\pi\,x)\,,\\
p_3(x,\,y) &=& -a \sin(\pi x)\,,
\end{eqnarray*}
and the velocity $\bm{q}_i(x,\,y) = - \nabla p_i(x,\,y)$ on each subdomain $\Omega_i$ for $i=1,2$ and $3$.
Both velocity and pressure functions are chosen in such a way that we have a $C^0$ continuity for the pressure
and for the normal component of the velocity on the curves $\Gamma_1$ and $\Gamma_2$, i.e.,
\begin{gather*}
    p_1= p_2\quad\text{and}\quad \bm{q}_1\cdot\bm{n}_1 +
    \bm{q}_2\cdot\bm{n}_1= 0\qquad\text{on}\,\,\Gamma_1\,,\\
    p_2= p_3\quad\text{and}\quad \bm{q}_2\cdot\bm{n}_2 +
    \bm{q}_3\cdot\bm{n}_2 = 0\qquad\text{on}\,\,\Gamma_2\,,
\end{gather*}
where $\bm{n}_{1}$ is the normal of $\Gamma_1$ pointing from $\Omega_1$ to
$\Omega_2$ and $\bm{n}_2$ is the normal of $\Gamma_2$ pointing from $\Omega_2$
to $\Omega_3$.

\begin{figure}[!htb]
\centering
\subfloat[Domain]{\includegraphics[width=0.3\textwidth]{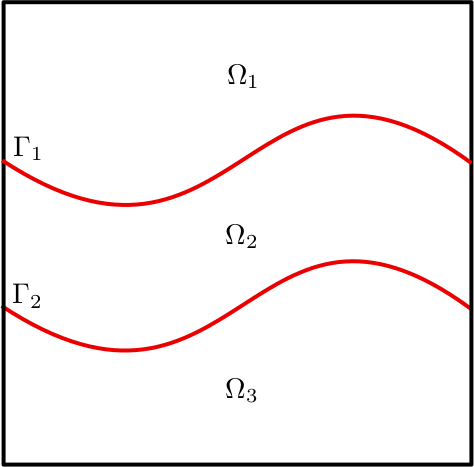}\label{fig:domExe3}}%
\hspace*{0.1\textwidth}%
\subfloat[Mesh]{\includegraphics[width=0.3\textwidth]{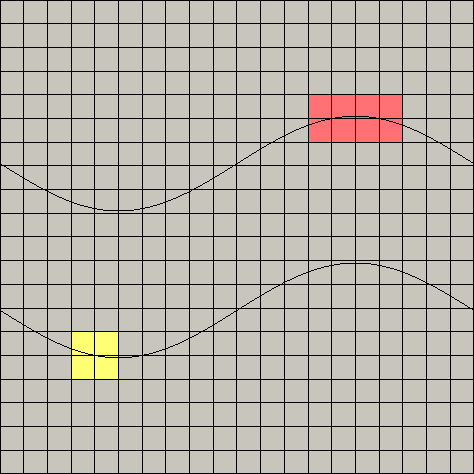}\label{fig:exe3Mesh}}%
\caption{Double internal curved interfaces. On the left, domain $\Omega$ considered in such example,
curved boundaries are highlighted in red. On the right, the whole mesh with the internal curved boundaries.
We show zooms of the yellow and red regions in Figure~\ref{fig:exe3MeshZoom}.}
\end{figure}

\paragraph{Meshes}
To generate the meshes, we follow the same idea as the example of
Subsection~\ref{sub:inside1}.  We build a background mesh composed by squares
and then we insert the curved internal interfaces, as shown in
Figure~\ref{fig:exe3Mesh}.  This is done, as previously, independently from the
background mesh, and thus the resulting meshes are composed by elements with
arbitrary size and shape, see Figure~\ref{fig:exe3MeshZoom}.  {Also in this
case, mesh element edges lying on the curvilinear interfaces exactly match the
interface for the \texttt{withGeo} approach, whereas they are approximated by
straight edges in the \texttt{noGeo} case.}

\begin{figure}[!htb]
\centering
\includegraphics[height=0.25\textwidth]{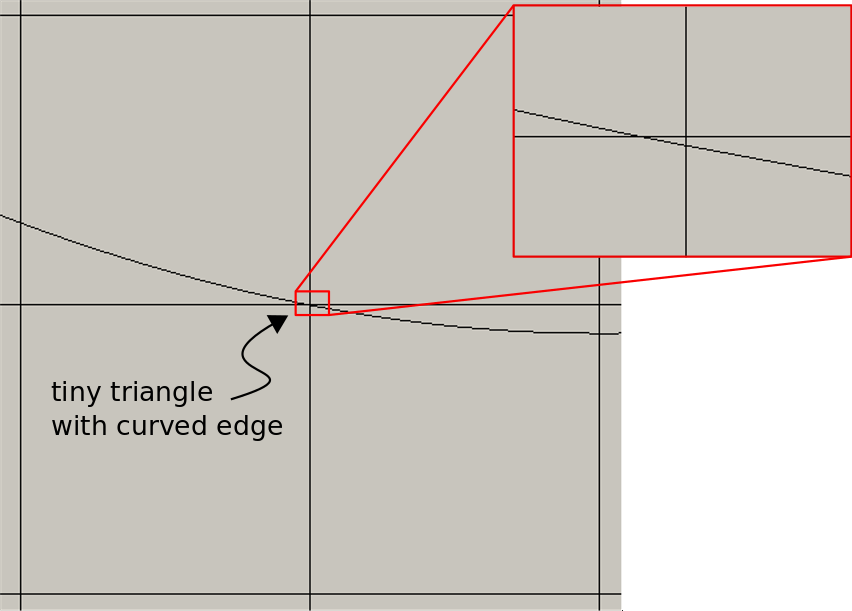}%
\hspace*{0.05\textwidth}%
\includegraphics[height=0.28\textwidth]{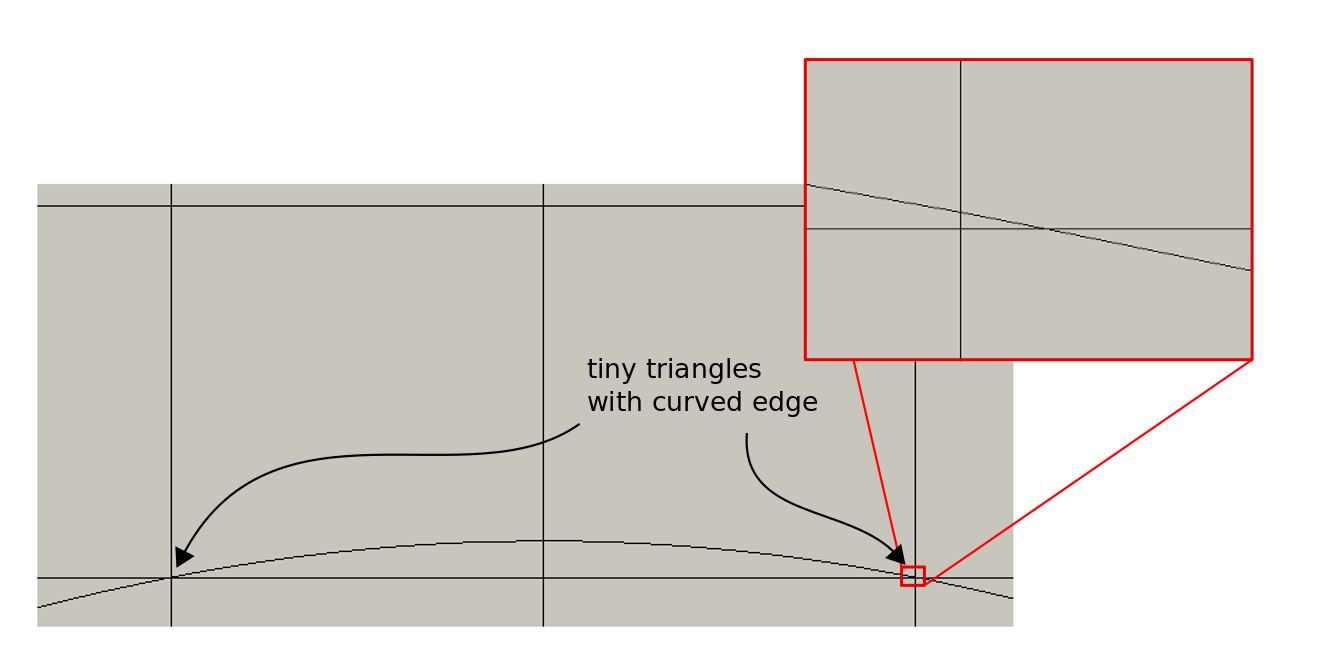}
\caption{Double internal curved interfaces: zooms of the yellow and red regions of Figure~\ref{fig:exe3Mesh},
where we highlight tiny triangles with a curved edge.}
\label{fig:exe3MeshZoom}
\end{figure}

%
%

\paragraph{Results}
In Figure~\ref{fig:exe3Res} we show convergence lines for both the \texttt{withGeo} and \texttt{noGeo} for values of $k=0,\ldots,4$.
The behaviour of error decay is again as expected: in the \texttt{noGeo} case error decay follows the expected trend for the used polynomial accuracy only for $k\leq1$, being, for $k>1$, always $\Or{2}$ for the prevailing effect of the geometrical error.
On the contrary, since appropriate basis functions are included in the definition of the approximation space in the proposed \texttt{withGeo} approach, optimal error decay is observed for the used polynomial accuracy level.

\begin{figure}[!htb]
\centering
\includegraphics[width=0.49\textwidth]{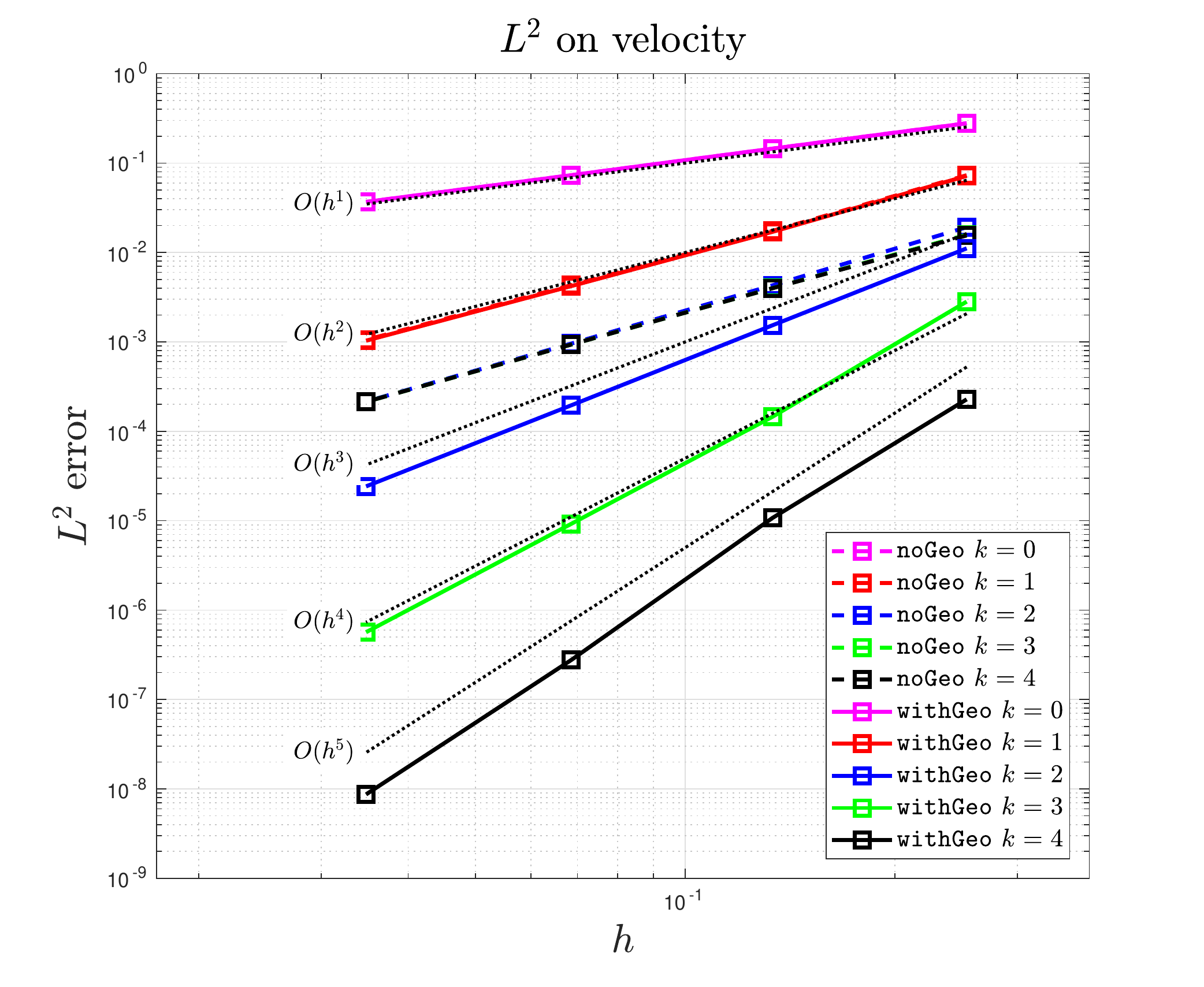}%
\includegraphics[width=0.49\textwidth]{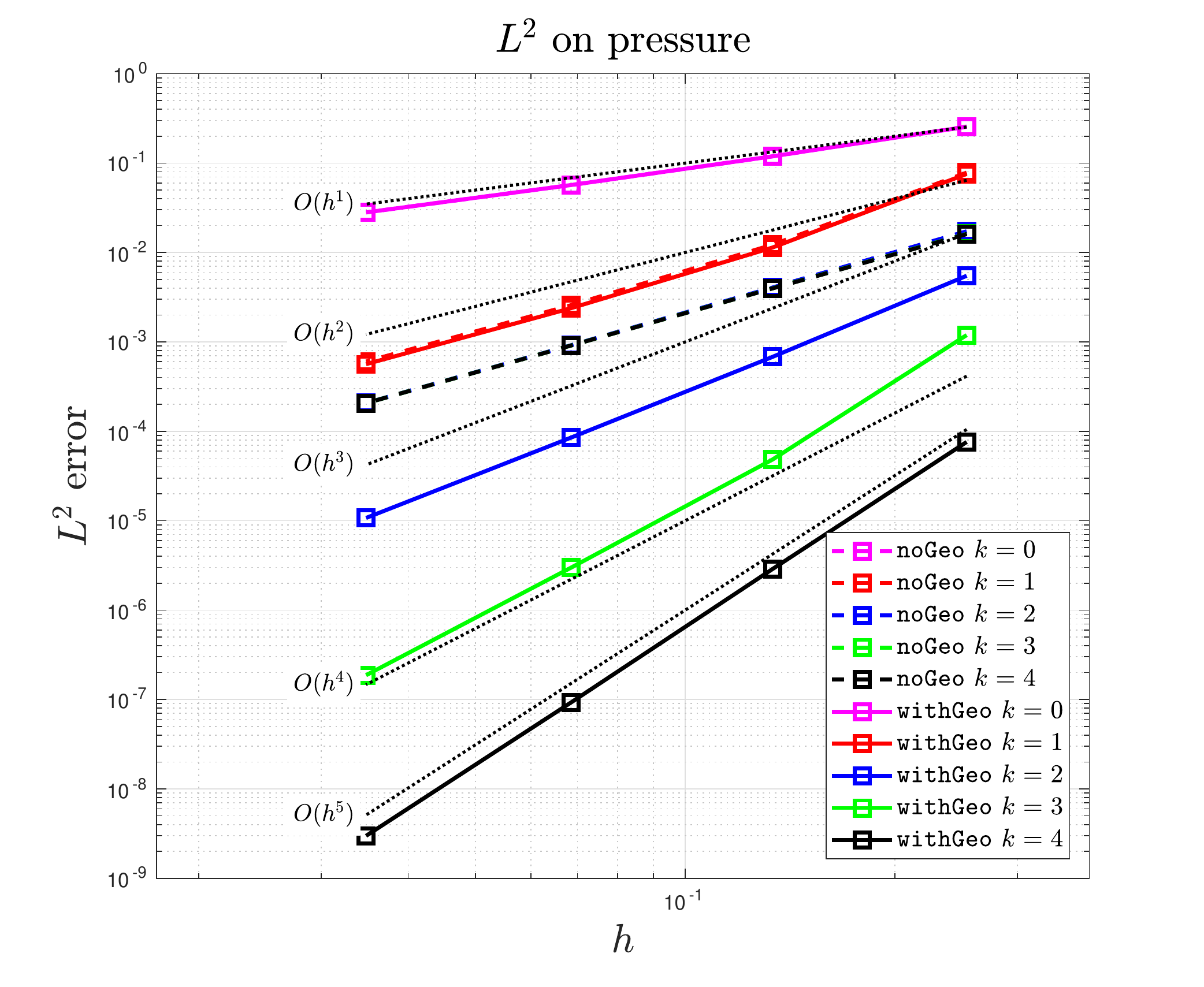}
\caption{Double internal curved interfaces: convergence lines for each VEM approximation degrees.}
\label{fig:exe3Res}
\end{figure}

\section{Conclusions}\label{sec:conclusion}

In this work we have performed a first analysis on the extension of the mixed
virtual element method to grids where elements might have curved edges, for
elliptic problems in 2D. A theoretical analysis is proposed to show
well-posedness of the discrete problem. A choice for the degrees of freedom
particularly well suited for discretizations on curvilinear edge elements is
highlighted, and a numerical scheme is proposed that handles in a coherent and
consistent way the geometry, thus exhibiting optimal error decay in accordance
to the polynomial accuracy level of the approximation.  This is particularly
suited for real applications where the geometrical error might dominate and
limit the accuracy of the numerical solution.  The numerical examples are in
accordance with the theoretical findings and showed the optimal error decay for
a domain with curved boundary and a domain with internal interfaces in contrast
with the standard mixed virtual element method where the geometrical error
jeopardizes the performances.  Natural extension of the current work are the
introduction of the mixed virtual element method for three-dimensional problems
with curved faces and for more general problems.

\section*{Acknowledgments}
The authors acknowledge financial support of INdAM-GNCS through project ``Bend VEM 3d'', 2020.
Author S.S. also acknowledges the financial support of MIUR through project ``Dipartimenti di
Eccellenza 2018-2022'' (Codice Unico di Progetto CUP E11G18000350001).

\bibliographystyle{plain}
\bibliography{biblio_no_url}

\end{document}